\documentclass[11pt, reqno]{amsart}

\usepackage{amsfonts, amssymb}
\usepackage{graphicx}
\usepackage{hyperref}
\usepackage{parskip}
\numberwithin{equation}{section}

\newtheorem{thm}{Theorem}[section]
\newtheorem{cor}[thm]{Corollary}
\newtheorem{lem}[thm]{Lemma}

\newtheorem{prop}[thm]{Proposition}

\theoremstyle{remark}

\theoremstyle{definition}
\newtheorem{defin}[thm]{Definition}

\newcommand{\Blo}{B_{m, \mathrm{lo \wedge lo}}}
\newcommand{\Bhi}{B_{m, \mathrm{hi \lor hi}}}

\newcommand{\Alol}{A_{\ell, \mathrm{lo \wedge lo}}}

\newcommand{\Nlo}{N_{\mathrm{lo}}}

\newcommand{\cN}{\mathcal{N}}

\newcommand{\Ec}{E_{\mathrm{crit}}}

\newcommand{\ds}{\; ds}
\newcommand{\R}{\mathbb{R}}
\newcommand{\C}{\mathbb{C}}
\newcommand{\RR}{\mathbb{R}^2}

\newcommand{\Sp}{\mathbb{S}}

\renewcommand{\Re}{\mathrm{Re}}
\renewcommand{\Im}{\mathrm{Im}}

\title[Global Schr\"odinger maps]{Global regularity of critical Schr\"odinger maps: 
subthreshold dispersed energy}
\author{Paul Smith}
\date{}
\address{University of California, Berkeley}
\email{smith@math.berkeley.edu}

\thanks{The author was supported by NSF grant DMS-1103877.}

\begin{document}

\begin{abstract}
We consider the energy-critical Schr\"odinger map initial value problem
with smooth initial data from $\R^2$ into the sphere $\Sp^2$.
Given sufficiently energy-dispersed data with subthreshold energy,
we prove that the system admits a unique global smooth solution.
This improves earlier analogous conditional results \cite{Sm10}.
The key behind this improvement lies in exploiting estimates on
the commutator of the Schr\"odinger map and harmonic map heat flows.
\end{abstract}

\maketitle

\tableofcontents

\section{Introduction} \label{S:Introduction}
We consider the
Schr\"odinger map initial value problem
\begin{equation}
\begin{cases}
\partial_t \phi &= \quad \phi \times \Delta \phi \\
\phi(x, 0) &= \quad \phi_0(x)
\end{cases}
\label{SM}
\end{equation}
with $\phi_0 : \R^d \to \Sp^2 \hookrightarrow \R^3$.
This system formally enjoys conservation of energy
\begin{equation}
E(\phi(t)) := \frac{1}{2} \int_{\mathbb{R}^d} \lvert \partial_x \phi(t) \rvert^2 dx
\label{Energy Def}
\end{equation}
and mass
\begin{equation*}
M(\phi(t)) := \int_{\mathbb{R}^d} \lvert \phi(t) - Q \rvert^2 dx
\end{equation*}
where $Q \in \Sp^2$ is fixed.
When $d = 2$, both (\ref{SM}) and (\ref{Energy Def}) are invariant
with respect to the scaling
\begin{equation}
\phi(x, t) \to \phi(\lambda x, \lambda^2 t), \quad \quad \lambda > 0
\label{E:scaling}
\end{equation}
and so when $d = 2$ the system (\ref{SM}) is called energy-critical.
This is the setting we consider in this article.

For the physical significance of (\ref{SM}), see
\cite{ChShUh00, NaStUh03-2, PaTo91, La67}.
The system may also be interpreted as a geometric analogue of the
free linear Schr\"odinger equation (see \cite[Introduction]{Sm10}).
For more general analogues of (\ref{SM}), e.g., for K\"ahler targets
other than $\mathbb{S}^2$, see \cite{DiWa01, Mc07, NaShVeZe07}.
See also \cite{KePoVe00, KeNa05, IoKe11} for connections with
other spin systems.

The local theory of \eqref{SM} is developed in \cite{SuSuBa86, ChShUh00, DiWa01, Mc07}.
Global wellposedness given data with small critical Sobolev norm (in all dimensions
$d \geq 2$) is shown in \cite{BeIoKeTa11}. A conditional extension of this result
that requires assuming an $L^4$ spacetime bound on the solution and 
smallness of a critical Besov norm of the data appears in \cite{Sm10}.
Recently it has been shown \cite{BeIoKeTa11b} that global existence and uniqueness
as well as scattering hold given 1-equivariant data with energy less than $4 \pi$.
In the radial setting (which excludes harmonic maps), \cite{GuKo11} establishes 
global wellposedness at any energy level.
In this article we pursue an extension of our work \cite{Sm10}, removing the $L^4$
boundedness condition.
To state our main results, we first introduce some Sobolev spaces.

For $\sigma \in [0, \infty)$,
let $H^\sigma = H^\sigma(\RR)$ denote the usual Sobolev space of complex-valued
functions on $\RR$.  For any $Q \in \mathbb{S}^2$, set
\begin{equation*}
H_Q^\sigma := \{ f : \RR \to \mathbb{R}^3 \text{ such that } |f(x)| \equiv 1 \text{ a.e. and }
f - Q \in H^\sigma \}
\end{equation*}
This is a metric space with induced distance $d_Q^\sigma(f, g) = \lVert f - g \rVert_{H^\sigma}$.
For $f \in H_Q^\sigma$ we set $\lVert f \rVert_{H_Q^\sigma} = d_Q^\sigma(f, Q)$ for short.
Define also
\begin{equation*}
H^\infty := \bigcap_{\sigma \in \mathbb{Z}_+} H^\sigma
\quad \text{and} \quad
H_Q^\infty := \bigcap_{\sigma \in \mathbb{Z}_+} H_Q^\sigma
\end{equation*}
Throughout $\mathbb{Z}_+ = \{0, 1, 2 \ldots \}$ denotes the nonnegative integers.

These definitions may be naturally extended to any spacetime slab $\RR \times [-T, T]$,
$T \in (0, \infty)$.
For any $\sigma, \rho \in \mathbb{Z}_+$,
let $H^{\sigma, \rho}(T)$ denote the Sobolev space of complex-valued functions on
$\RR \times [-T, T]$ with the norm
\begin{equation*}
\lVert f \rVert_{H^{\sigma, \rho}(T)} := \sup_{t \in (-T, T)}
\sum_{\rho^\prime = 0}^\rho \lVert \partial_t^{\rho^\prime} f(\cdot, t)\rVert_{H^\sigma}
\end{equation*}
and for $Q \in \mathbb{S}^2$ endow
\begin{equation*}
H_{Q}^{\sigma, \rho} := \{ f : \RR \times [-T, T] \to \mathbb{R}^3
\text{ such that } | f(x, t) | \equiv 1
\text{ a.e. and } f - Q \in H^{\sigma, \rho}(T) \}
\end{equation*}
with the metric induced by the $H^{\sigma, \rho}(T)$ norm.  Also, define the spaces
\begin{equation*}
H^{\infty, \infty}(T) = \bigcap_{\sigma, \rho \in \mathbb{Z}_+} H^{\sigma, \rho}(T)
\quad \text{and} \quad
H_Q^{\infty, \infty}(T) = \bigcap_{\sigma, \rho \in \mathbb{Z}_+} H_Q^{\sigma, \rho}(T)
\end{equation*}
For $f \in H^\infty$ and $\sigma \geq 0$ 
we define the homogeneous Sobolev norms as
\begin{equation*}
\lVert f \rVert_{\dot{H}^\sigma} =
\lVert \hat{f}(\xi) \cdot \lvert \xi \rvert^\sigma \rVert_{L^2}
\end{equation*}
Here $\hat{f}$ stands for the Fourier transform in the spatial variables only.
We similarly define $\dot{H}^\sigma_Q$ by first translating $f$ by $Q$.

We now state our main global result.
\begin{thm}[Global regularity]
Fix $Q \in \mathbb{S}^2$. Then there exists $\varepsilon_0 > 0$ such that for all $\phi_0 \in H_Q$ with
$E_0 := E(\phi_0) < \Ec$ and
\begin{equation*}
\sup_{k \in \mathbb{Z}} \lVert P_k \partial_x \phi_0 \rVert_{L^2_x} \leq \varepsilon_0,
\end{equation*}
equation (\ref{SM}) admits a unique global solution $\phi \in C(\mathbb{R} \to H_Q^\infty)$.
\label{T:Main}
\end{thm}
The subthreshold assumption on the initial data is important only insofar as it is used
to establish certain bounds on the gauge. It may be possible to establish such bounds
without making use of the subthreshold assumption, though we do not pursue this here.

\begin{thm}[Uniform bounds]
Fix $Q \in \mathbb{S}^2$. Let $\sigma_1 \geq 1$.
Then there exists $\tilde{\varepsilon}_0(\sigma_1) \in (0, \varepsilon_0]$ such that for all $\phi_0 \in H_Q$ with
$E_0 := E(\phi_0) < \Ec$ and
\begin{equation*}
\sup_{k \in \mathbb{Z}} \lVert P_k \partial_x \phi_0 \rVert_{L^2_x} \leq \tilde{\varepsilon}_0,
\end{equation*}
the global solution $\phi$ constructed in Theorem \ref{T:Main} satisfies the uniform
bounds
\begin{equation*}
\sup_{t \in \mathbb{R}} \lVert \phi(t) - Q \rVert_{H^\sigma}
\lesssim_\sigma
\lVert \phi_0 - Q \rVert_{H^\sigma},
\quad
1 \leq \sigma \leq \sigma_1
\end{equation*}
\label{T:Uniform}
\end{thm}
To prove Theorems \ref{T:Main} and \ref{T:Uniform}, we first establish estimates at fixed
dyadic scales. In doing so we make use of an important technical tool, frequency envelopes in particular, defined as follows.
\begin{defin}[Frequency envelopes]
Let $\delta > 0$ be fixed.
A positive sequence $\{a_k\}_{k \in \mathbb{Z}}$ is a \emph{frequency envelope} if it belongs to $\ell^2$
and is slowly varying:
\begin{equation}
a_k \leq a_j 2^{\delta \lvert k - j \rvert},
\quad \quad j, k \in \mathbb{Z}
\label{Slowly Varying}
\end{equation}
A frequency envelope $\{a_k\}_{k \in \mathbb{Z}}$ is $\varepsilon$-energy dispersed if it satisfies
the additional condition
\begin{equation*}
\sup_{k \in \mathbb{Z}} a_k \leq \varepsilon
\end{equation*}
\end{defin}
Note in particular that frequency envelopes satisfy the following summation rules:
\begin{align}
\sum_{k^\prime \leq k} 2^{p k^\prime} a_{k^\prime}
&\lesssim (p - \delta)^{-1} 2^{p k} a_k
&p > \delta \label{Sum 1}\\
\sum_{k^\prime \geq k} 2^{-p k^\prime} a_{k^\prime}
&\lesssim (p - \delta)^{-1} 2^{-p k} a_k
&p > \delta \label{Sum 2}
\end{align}
In practice we work with $p$ bounded away from $\delta$, e.g., $p > 2 \delta$ suffices,
and iterate these inequalities only $O(1)$ times.
Therefore in applications we drop the factors $(p - \delta)^{-1}$ appearing in
(\ref{Sum 1}) and (\ref{Sum 2}).

Given initial data $\phi_0 \in H_Q^\infty$, define
for all $\sigma \geq 0$ and $k \in \mathbb{Z}$
\begin{equation}
c_k(\sigma) 
:= 
\sup_{k^\prime \in \mathbb{Z}} 
2^{-\delta \lvert k - k^\prime \rvert} 2^{\sigma k^\prime}
\lVert P_{k^\prime} \partial_x \phi_0 \rVert_{L_x^2}
\label{c Envelope}
\end{equation}
and set $c_k := c_k(0)$ for short.
For $\sigma \in [0, \sigma_1]$ it then holds that
\begin{equation}
\lVert \partial_x \phi_0 \rVert_{\dot{H}_x^\sigma}^2 \sim
\sum_{k \in \mathbb{Z}} c_k^2(\sigma)
\quad\quad \text{and} \quad\quad
\lVert P_k \partial_x \phi_0 \rVert_{L_x^2} \leq c_k(\sigma) 2^{-\sigma k}
\label{c cons}
\end{equation}
Similarly, for $\phi \in H_Q^{\infty, \infty}(T)$, define for all
$\sigma \geq 0$ and $k \in \mathbb{Z}$
\begin{equation}
\gamma_k(\sigma) := \sup_{k^\prime \in \mathbb{Z}} 2^{-\delta \lvert k - k^\prime \rvert}
2^{\sigma k^\prime} \lVert P_{k^\prime} \phi \rVert_{L_t^\infty L_x^2}
\label{gamma Envelope}
\end{equation}
and set $\gamma_k := \gamma_k(1)$.

For technical reasons we construct a solution $\phi$ on a time
interval $(-2^{2 \mathcal{K}}, 2^{2 \mathcal{K}})$ for some given $\mathcal{K} \in \mathbb{Z}_+$
and then proceed to prove bounds that are uniform in $\mathcal{K}$.  We assume
$1 \ll \mathcal{K} \in \mathbb{Z}_+$ is chosen and hereafter fixed.
By the local theory, we may assume that we have a solution 
$\phi \in C([-T, T] \to H_Q^\infty)$ of (\ref{SM}) on the time interval $[-T, T]$ for some
$T \in (0, 2^{2 \mathcal{K}}]$. 
In order to extend $\phi$ to a solution on all of
$(-2^{2 \mathcal{K}}, 2^{2 \mathcal{K}})$ with uniform bounds 
(uniform in $T, \mathcal{K}$)
it suffices
to prove
uniform a priori estimates on
\begin{equation*}
\sup_{t \in [-T, T]} \lVert \phi(t) \rVert_{H_Q^\sigma}
\end{equation*}
for, say, $\sigma$ in the interval $[1, \sigma_1]$, with $\sigma_1 \gg 1$ chosen sufficiently large.

The first step in our approach, carried out in \S \ref{S:Gauge Field Equations},
is to lift the Schr\"odinger map system (\ref{SM})
to the tangent bundle and view it with respect to the caloric gauge.
The lift of (\ref{SM}) becomes a system of coupled magnetic nonlinear Schr\"odinger equations
of the form
\begin{equation}
(i \partial_t + \Delta) \psi_m = B_m + V_m, \quad \quad m = 1, 2 
\label{dSM}
\end{equation}
with initial data $\psi_1(0), \psi_2(0)$.  Here $B_m$ and $V_m$ respectively denote
the magnetic and electric potentials (see (\ref{B Def}) and (\ref{V Def}) for
definitions).
At the level of the tangent bundle the goal is to prove
a priori bounds on $\lVert \psi_m \rVert_{L_t^\infty H_x^\sigma}$,
which we establish by proving stronger frequency-localized estimates.
The proof of Theorem \ref{T:Main} is then completed by transfering
bounds on the derivative fields $\psi_m$ back to bounds on the map $\phi$.

Energy spaces are not sufficient for controlling $P_k \psi_m$, and so we combine
in addition to these a host of Strichartz, local smoothing, and maximal function type spaces into one
space $G_k(T)$ (see \S \ref{S:Function spaces} for precise definitions).
As our goal will be to express control of the $G_k(T)$ norms of $P_k \psi_m$ 
in terms of the energy of the frequency localizations of the initial data,
we introduce the following frequency envelopes.
Let $\sigma_1 \in \mathbb{Z}_+$ be positive.
For $\sigma \in [0, \sigma_1 - 1]$, set
\begin{equation}
b_k(\sigma) := \sup_{k^\prime \in \mathbb{Z}} 2^{\sigma k^\prime} 2^{-\delta \lvert k - k^\prime \rvert}
\lVert P_{k^\prime} \psi_x \rVert_{G_k(T)}
\label{b Envelope}
\end{equation}
By (\ref{Soft Field Space Bounds}) 
these envelopes are finite and in $\ell^2$.
We abbreviate $b_k(0)$ by setting $b_k := b_k(0)$.

We now state the key result for solutions of the gauge field equation (\ref{dSM}) (with the caloric
gauge as the gauge choice).
\begin{thm}
Assume $T \in (0, 2^{2 \mathcal{K}}]$ and $Q \in S^2$.
Choose $\sigma_1 \in \mathbb{Z}_+$ positive.
Let $\varepsilon_0 > 0$
and let $\phi \in H_Q^{\infty, \infty}(T)$
be a solution of the Schr\"odinger map system (\ref{SM}) whose initial data $\phi_0$
has energy $E_0 := E(\phi_0) < E_{\mathrm{crit}}$ and satisfies the energy dispersion condition
\begin{equation}
\sup_{k \in \mathbb{Z}} c_k \leq \varepsilon_0
\label{iED}
\end{equation}
Suppose that the bootstrap hypothesis
\begin{equation}
b_k \leq \varepsilon_0^{-1/10} c_k
\label{Main Bootstrap}
\end{equation}
is satisfied.
Then, for $\varepsilon_0$ sufficiently small,
\begin{equation}
b_k(\sigma) \lesssim c_k(\sigma)
\label{Field Conclusion}
\end{equation}
holds for all $\sigma \in [0, \sigma_1 - 1]$ and $k \in \mathbb{Z}$.
\label{T:EnvelopeTheorem}
\end{thm}
We use a continuity argument to prove Theorem \ref{T:EnvelopeTheorem}.
For $T^\prime \in (0, T]$, let
\begin{equation*}
\Psi(T^\prime) = \sup_{k \in \mathbb{Z}} c_k^{-1}
\lVert P_k \psi_m(s = 0) \rVert_{G_k(T^\prime)}
\end{equation*}
Then $\Psi : (0, T] \to [0, \infty)$ is well-defined, increasing, continuous, and satisfies
\begin{equation*}
\lim_{T^\prime \to 0} \Psi(T^\prime) \lesssim 1
\end{equation*}
The critical implication to establish is
\begin{equation*}
\Psi(T^\prime) \leq \varepsilon_1^{-1/10} \implies
\Psi(T^\prime) \lesssim 1
\end{equation*}
which in particular follows from
\begin{equation}
b_k \lesssim c_k
\label{Goal}
\end{equation}
We must also similarly establish
\begin{equation}
b_k(\sigma) \lesssim c_k(\sigma)
\label{Goal2}
\end{equation}
for $\sigma \in (0, \sigma_1 - 1]$ in order to bring under control the higher-order
Sobolev norms.
The bounds (\ref{Goal2}), however, will be seen to follow as an easy
consequence of the proof of (\ref{Goal}), and therefore
throughout this article the emphasis will be upon proving (\ref{Goal}).

\begin{cor}
Given the conditions of Theorem \ref{T:EnvelopeTheorem},
\begin{equation}
\lVert P_k \lvert \partial_x \rvert^\sigma \partial_m \phi 
\rVert_{L_t^\infty L_x^2( (-T, T) \times \RR)}
\lesssim c_k(\sigma)
\label{Main Conclusion}
\end{equation}
holds for all $\sigma \in [0, \sigma_1 - 1]$.
\label{C:ETcor}
\end{cor}
In view of the local theory and (\ref{c cons}), this corollary
implies Theorems \ref{T:Main} and \ref{T:Uniform}.

One strategy for proving
(\ref{Goal}), i.e., $b_k \lesssim c_k$, is to show that the nonlinearity is perturbative.
Along these lines, we project
(\ref{NLS}) to frequencies $\sim 2^k$
and apply the linear estimate of
Proposition \ref{MainLinearEstimate}:
\begin{equation}
\lVert P_k \psi_m \rVert_{G_k(T)} \lesssim \lVert P_k \psi_m (0) \rVert_{L_x^2} +
\lVert P_k \cN_m \rVert_{N_k(T)}
\label{main-estimate}
\end{equation}
If the nonlinearity is perturbative in the sense that 
$\lVert P_k \cN_m \rVert_{N_k(T)} \lesssim \varepsilon b_k$,
then (\ref{main-estimate}) implies $b_k \lesssim c_k + \varepsilon b_k$, which establishes
(\ref{Goal}). This is the general approach of \cite{BeIoKeTa11} in the case of small critical
norm. In the dispersed setting, the entirety of $P_k \cN_m$ does not obey such a bound; 
it fails to be perturbative in this sense. Nevertheless, several
pieces of $\cN_m$ are perturbative, and showing this
plays a crucial role in enabling us to gain control over the remaining 
nonperturbative part of the nonlinearity, as in \cite{Sm10}.
The key there was a certain bilinear Strichartz estimate adapted to a magnetic NLS.
In this article we are able to establish stronger bounds on the electric potential $V_m$,
which then puts us into a position to apply the machinery of \cite{Sm10} but without
the conditional $L^4$ condition required there. Hence the most important technical contribution of this article lies in
\S \ref{SS:Electric}. This improvement also leads to some slight technical simplifications over \cite{Sm10}
in establishing \eqref{Goal2} for $\sigma > 0$. We remark that it would be interesting to
extend the results of \cite{Sm10} in a direction other than the one we take in this article;
in particular, one expects global wellposedness under the assumption of an $L^4$-type bound
even if the data is not energy-dispersed.

\section{Gauge selection} \label{S:Gauge Field Equations}

We begin with some
constructions valid for any smooth function $\phi : \RR \times (-T, T) \to \mathbb{S}^2$.
Space and time derivatives of $\phi$ are denoted by
 $\partial_\alpha \phi(x, t)$, where $\alpha = 1, 2, 3$ ranges
over the spatial variables $x_1, x_2$ and time $t$ with $\partial_3 = \partial_t$.

Select a smooth orthonormal frame
$(v(x, t), w(x, t))$ for $T_{\phi(x, t)} \mathbb{S}^2$, i.e.,
smooth functions $v, w: \RR \times (-T, T) \to T_{\phi(x, t)} \mathbb{S}^2$
such that for each point $(x, t)$ in the domain the vectors
$v(x, t), w(x, t)$ form an orthonormal basis of $T_{\phi(x, t)} \mathbb{S}^2$.
As a matter of convention we assume that $v$ and $w$ are chosen so that
$v \times w = \phi$.

We define the so-called derivative fields $\psi_\alpha$
by writing $\partial_\alpha \phi$ with respect to $(v, w)$ and 
identifying $\R^2$ with $\C$:
\begin{equation}
\psi_\alpha := v \cdot \partial_\alpha \phi + i w \cdot \partial_\alpha \phi
\label{Derivative Field}
\end{equation}
In other words, $\psi_\alpha$ is determined by
\begin{equation}
\partial_\alpha \phi = v \; \Re(\psi_\alpha) + w \; \Im(\psi_\alpha)
\label{Phi Frame Rep}
\end{equation}
Through this identification,
the Levi-Civita connection on $\mathbb{S}^2$ pulls back to
a covariant derivative on $\mathbb{C}$, denoted by
\begin{equation*}
D_\alpha := \partial_\alpha + i A_\alpha
\end{equation*}
where the real-valued connection coefficients $A_\alpha$
are determined via
\begin{equation}
A_\alpha = w \cdot \partial_\alpha v
\label{Connection Coefficient}
\end{equation}
Due to the fact that the Levi-Civita connection on $\mathbb{S}^2$ is torsion-free,
the derivative fields satisfy the compatibility relations
\begin{equation}
D_\beta \psi_\alpha = D_\alpha \psi_\beta
\label{Torsion Free}
\end{equation}
A straightforward calculation (which uses the fact that the sphere
has constant curvature) shows
\begin{equation*}
\partial_\beta A_\alpha - \partial_\alpha A_\beta =
\Im( \psi_\beta \overline{\psi_\alpha}) =: q_{\beta \alpha}
\end{equation*}
so that the curvature of the connection is given by
\begin{equation}
[D_\beta, D_\alpha] :=
D_\beta D_\alpha - D_\alpha D_\beta = i q_{\beta \alpha}
\label{Curvature}
\end{equation}
Suppose $\phi$ is a smooth solution of the Schr\"odinger map
system (\ref{SM}). This system lifts to
\begin{equation}
\psi_t = i D_\ell \psi_\ell
\label{Psi_t Equation}
\end{equation}
because
\begin{equation*}
\phi \times \Delta \phi = J(\phi) (\phi^* \nabla)_j \partial_j \phi
\end{equation*}
where $J(\phi)$ denotes the complex structure $\phi \times$
and $(\phi^* \nabla)_j$ the pullback of the Levi-Civita connection
$\nabla$ on the sphere. Here we are using the convention that repeated Roman indices 
are summed over the spatial variables. In fact, throughout we consistently use Roman indices
only for spatial variables and Greek indices for variables that are either spatial or temporal.

Using (\ref{Torsion Free}) and (\ref{Curvature}) in (\ref{Psi_t Equation}) yields
\begin{equation*}
D_t \psi_m = i D_\ell D_\ell \psi_m + q_{\ell m} \psi_\ell
\end{equation*}
which is equivalent to the nonlinear Schr\"odinger equation
\begin{equation}
(i \partial_t + \Delta) \psi_m =
\cN_m
\label{NLS}
\end{equation}
with nonlinearity $\cN_m$ defined by the formula
\begin{equation*}
\cN_m 
:= 
-i A_\ell \partial_\ell \psi_m - i \partial_\ell (A_\ell \psi_m)
+ (A_t + A_x^2) \psi_m - i \psi_\ell \Im(\overline{\psi_\ell} \psi_m)
\end{equation*}
We split this nonlinearity as a sum $\cN_m = B_m + V_m$ with $B_m$ and $V_m$ defined by
\begin{equation}
B_m 
:= - i \partial_\ell (A_\ell \psi_m) - i A_\ell \partial_\ell \psi_m \label{B Def}
\end{equation}
and
\begin{equation}
V_m := (A_t + A_x^2) \psi_m - i \psi_\ell \Im(\overline{\psi_\ell} \psi_m)
\label{V Def}
\end{equation}
We call $B_m$ the magnetic part of the potential and $V_m$ the electric part.
The differentiated Schr\"odinger map system then takes the form
\begin{equation}
\left\{
\begin{array}{ll}
(i \partial_t + \Delta) \psi_m &= B_m + V_m \\
D_\alpha \psi_\beta &= D_\beta \psi_\alpha \\
\partial_\alpha A_\beta - \partial_\beta A_\alpha &= \Im(\psi_\alpha \overline{\psi_\beta})
\end{array}
\right.
\label{DSM System}
\end{equation}
A solution $\psi_m$ of (\ref{DSM System}) cannot be determined uniquely without
first choosing an orthonormal frame $(v, w)$.  Changing a given choice
of orthonormal frame induces a gauge transformation that modifies
the derivative fields and connection coefficients according to
\begin{equation*}
\psi_m \to e^{-i \theta} \psi_m,
\quad\quad
A_m \to A_m + \partial_m \theta
\end{equation*}
The system (\ref{DSM System}) is invariant
with respect to such gauge transformations.

Gauges were first used to study (\ref{SM}) in \cite{ChShUh00}.
The caloric gauge, which is the gauge we select,
was introduced by Tao in \cite{Trenorm} in the setting of wave maps
into hyperbolic space.
In the series of unpublished papers \cite{T3, T4, T5, T6, T7},
Tao used this gauge in establishing
global regularity of wave maps
into hyperbolic space.  
In his unpublished note \cite{TSchroedinger},
Tao also suggested the caloric gauge
as an alternative to the Coulomb gauge in studying Schr\"odinger maps.
The caloric gauge was first used in the Schr\"odinger maps problem
by Bejenaru, Ionescu, Kenig, and Tataru in \cite{BeIoKeTa11} to establish global well-posedness
in the setting of initial data with sufficiently small critical norm.
We recommend \cite{Trenorm, T4, TSchroedinger, BeIoKeTa11} for background
on the caloric gauge and for helpful heuristics.

\begin{thm}[The caloric gauge]
Let $T \in (0, \infty)$, $Q \in \Sp^2$, and let $\phi(x, t) \in H_{Q}^{\infty, \infty}(T)$
be such that $\sup_{t \in (-T, T)} E(\phi(t)) < E_{\mathrm{crit}}$.
Then
there exists a unique smooth extension 
$\phi(s, x, t) \in C([0, \infty) \to H_Q^{\infty, \infty}(T))$
solving the covariant heat equation
\begin{equation}
\partial_s \phi = \Delta \phi + \phi \cdot \lvert \partial_x \phi \rvert^2
\label{Covariant Heat}
\end{equation}
and with $\phi(0, x, t) = \phi(x, t)$.  Moreover, 
for any given choice of a (constant) orthonormal basis $(v_\infty, w_\infty)$
of $T_Q \Sp^2$,
there exist smooth functions
$v, w: [0, \infty) \times \R^2 \times (-T, T) \to \Sp^2$ such that at each point $(s, x, t),$ the set
$\{v, w, \phi\}$ naturally forms an orthonormal basis for $\R^3$, the
gauge condition
\begin{equation}
w \cdot \partial_s v \equiv 0
\label{Gauge Condition}
\end{equation}
is satisfied, and
\begin{equation}
\lvert \partial_x^\rho f(s) \rvert \lesssim_\rho \langle s \rangle^{-(\lvert \rho \rvert + 1)/2}
\label{Soft Decay Bounds}
\end{equation}
for each $f \in \{ \phi - Q, v - v_\infty, w - w_\infty \}$, multiindex $\rho$, and $s \geq 0$.
\end{thm}

We now record a couple of the technical bounds from \cite{Sm09}, which will
be useful later on in controlling terms along the heat flow.
\begin{thm}
The following bounds hold:
\begin{align}
\int_0^\infty
\sup_{x \in \mathbb{R}^2} \lvert \psi_x(s, x) \rvert^2
ds
&\lesssim_{E_0} 1
\label{ek4}
\\
\lVert A_x(s) \rVert_{L^2_x(\mathbb{R}^2)}
&\lesssim_{E_0} 1
\label{A_x Bound}
\end{align}
\end{thm}
\begin{proof}
For (\ref{ek4}), see \cite[\S 4]{Sm09}.
For (\ref{A_x Bound}), see \cite[\S \S7, 7.1]{Sm09}.
\end{proof}

From now on we assume that our Schr\"odinger map $\phi = \phi(x, t)$ has
been extended via harmonic map heat flow to a map $\tilde{\phi}(s, x, t), s \geq 0$.
From here on out we also call the extension $\phi$.
Let $\partial_0 = \partial_s$ and define for all
$(s, x, t) \in [0, \infty) \times \RR \times (-T, T)$ the various gauge components
\begin{align*}
\psi_\alpha &:= v \cdot \partial_\alpha \phi + i w \cdot \partial_\alpha \phi \\
A_\alpha &:= w \cdot \partial_\alpha v \\
D_\alpha &:= \partial_\alpha + A_\alpha \\
q_{\alpha \beta} &:= \partial_\alpha A_\beta - \partial_\beta A_\alpha
\end{align*}
The parallel transport condition $w \cdot \partial_s v \equiv 0$ is
equivalently expressed in terms of the connection coefficients as
\begin{equation}
A_s \equiv 0
\label{A_s}
\end{equation}
Expressed in terms of the gauge, the heat flow (\ref{Covariant Heat})
lifts to
\begin{equation}
\psi_s = D_\ell \psi_\ell
\label{Psi_s Equation}
\end{equation}
Using (\ref{Torsion Free}) and (\ref{Curvature}), we may take
the $D_m$ covariant derivative of (\ref{Psi_s Equation}) and rewrite it as
\begin{equation*}
\partial_s \psi_m = D_\ell D_\ell \psi_m + i \Im(\psi_m \overline{\psi_\ell}) \psi_\ell
\end{equation*}
or equivalently
\begin{equation}
(\partial_s - \Delta) \psi_m =
i A_\ell \partial_\ell \psi_m +
i \partial_\ell (A_\ell \psi_m) -
A_x^2 \psi_m + i \psi_\ell \Im(\overline{\psi_\ell} \psi_m)
\label{HF}
\end{equation}
More generally, taking the $D_\alpha$ covariant derivative, we obtain
\begin{equation}
(\partial_s - \Delta) \psi_\alpha =
U_\alpha
\label{genPsiEQ}
\end{equation}
with
\begin{equation}
U_\alpha 
:=
i A_\ell \partial_\ell \psi_\alpha + i \partial_\ell(A_\ell \psi_\alpha) 
- A_x^2 \psi_\alpha + i \psi_\ell \Im(\overline{\psi_\ell} \psi_\alpha) 
\label{Heat Nonlinearity 0}
\end{equation}
From (\ref{Curvature}) and (\ref{A_s}) it follows that
\begin{equation*}
\partial_s A_\alpha = \Im(\psi_s \overline{\psi_\alpha})
\end{equation*}
and integrating back from $s = \infty$ (justified using (\ref{Soft Decay Bounds})) yields
\begin{equation}
A_\alpha(s) = - \int_s^\infty \Im(\overline{\psi_\alpha} \psi_s)(s^\prime) \ds^\prime
\label{CC Integral Rep}
\end{equation}

At $s = 0$, $\phi$ satisfies both (\ref{SM}) and (\ref{Covariant Heat}),
or equivalently, $\psi_t(s = 0) = i \psi_s(s = 0)$. While for $s > 0$
it continues to be the case that $\psi_s = D_\ell \psi_\ell$ by construction,
we no longer necessarily have $\psi_t(s) = i D_\ell(s) \psi_\ell(s)$, i.e.,
$\phi(s, x, t)$ is not necessarily a Schr\"odinger map at fixed $s > 0$. 
We recall from \cite[\S2]{Sm10} the following description of their commutator
$\Psi = \psi_t - i \psi_s$.
\begin{lem}[Flows do not commute]
Set $\Psi := \psi_t - i \psi_s$. Then
\[
\begin{split}
\partial_s \Psi
&= D_\ell D_\ell \Psi + i \Im(\psi_t \bar{\psi}_\ell) \psi_\ell - \Im(\psi_s \bar{\psi}_\ell) \psi_\ell 
\\
&= D_\ell D_\ell \Psi + i \Im(\Psi \bar{\psi}_\ell) \psi_\ell + i \Im(i\psi_s \bar{\psi}_\ell) \psi_\ell
- \Im(\psi_s \bar{\psi}_\ell) \psi_\ell
\end{split}
\]
\label{L:Commutator}
\end{lem}
We now record some frequency-localized energy estimates
which find application in controlling the paralinearized nonlinearity.
\begin{thm}[Frequency-localized energy estimates for heat flow]
Let $f \in \{\phi, v, w\}$.  Then for $\sigma \in [1, \sigma_1]$ the bound
\begin{equation}
\lVert P_k f(s) \rVert_{L_t^\infty L_x^2}
\lesssim
2^{-\sigma k} \gamma_k(\sigma) (1 + s 2^{2k})^{-20}
\label{Hard Envelope Bounds}
\end{equation}
holds,
and, for any $\sigma, \rho \in \mathbb{Z}_+$, it holds that
\begin{equation}
\sup_{k \in \mathbb{Z}} \sup_{s \in [0, \infty)}
(1 + s)^{\sigma / 2} 2^{\sigma k} \lVert P_k \partial_t^\rho f(s) \rVert_{L_t^\infty L_x^2} < \infty
\label{Soft Envelope Bounds}
\end{equation}
\label{Energy Envelope Decay}
\end{thm}

\begin{cor}[Frequency-localized energy estimates for the caloric gauge]
For $\sigma \in [0, \sigma_1 - 1]$ it holds that
\begin{equation}
\lVert P_k \psi_x(s) \rVert_{L_t^\infty L_x^2} +
\lVert P_k A_m(s) \rVert_{L_t^\infty L_x^2}
\lesssim
2^k 2^{-\sigma k} \gamma_k(\sigma) (1 + s 2^{2k})^{-20}
\label{Hard Field Bounds}
\end{equation}
Moreover, for any $\sigma \in \mathbb{Z}_+$,
\begin{equation}
\sup_{k \in \mathbb{Z}} \sup_{s \in [0, \infty)} (1 + s)^{\sigma / 2}
2^{\sigma k} 2^{-k} \left(
\lVert P_k (\partial_t^\rho \psi_x(s)) \rVert_{L_t^\infty L_x^2} +
\lVert P_k (\partial_t^\rho A_x(s)) \rVert_{L_t^\infty L_x^2} \right)
< \infty
\label{Soft Field Space Bounds}
\end{equation}
and
\begin{equation}
\sup_{k \in \mathbb{Z}} \sup_{s \in [0, \infty)} (1 + s)^{\sigma / 2}
2^{\sigma k} \left(
\lVert P_k (\partial_t^\rho \psi_t(s)) \rVert_{L_t^\infty L_x^2} +
\lVert P_k (\partial_t^\rho A_t(s)) \rVert_{L_t^\infty L_x^2} \right)
< \infty
\label{Soft Field Time Bounds}
\end{equation}
\label{EED Cor}
\end{cor}
Full proofs of Theorem \ref{Energy Envelope Decay} and its corollary may be found in
\cite{Sm10}. These are extensions to the energy-dispersed setting
of analogous small-energy bounds from \cite{BeIoKeTa11}.

\section{Function spaces} \label{S:Function spaces}

\subsection{Definitions}

\begin{defin}[Littlewood-Paley multipliers]
Let $\eta_0 : \mathbb{R} \to [0, 1]$ be a smooth even function vanishing outside the interval
$[-8/5, 8/5]$ and equal to $1$ on $[-5/4, 5/4]$.  For $j \in \mathbb{Z}$, set
\begin{equation*}
\chi_j(\cdot) = \eta_0( \cdot / 2^j ) - \eta_0( \cdot / 2^{j - 1}),
\quad
\chi_{\leq j}(\cdot) = \eta_0(\cdot / 2^j)
\end{equation*}
Let $P_k$ denote the operator on $L^\infty(\RR)$ defined by the Fourier multiplier
$\xi \to \chi_k(\lvert \xi \rvert)$.  For any interval $I \subset \mathbb{R}$, let
$\chi_I$  be the Fourier multiplier defined by $\chi_I = \sum_{j \in I \cap \mathbb{Z}} \chi_j$ and 
let $P_I$ denote its corresponding operator on $L^\infty(\RR)$.  We denote
$P_{(-\infty, k]}$ by $P_{\leq k}$ for short.  For $\theta \in \mathbb{S}^1$ and $k \in \mathbb{Z}$, we define
the operators $P_{k, \theta}$ by the Fourier multipliers $\xi \to \chi_k(\xi \cdot \theta)$.
\label{D:LP}
\end{defin}

Some frequency interactions in the nonlinearity of (\ref{NLS}) can be controlled
using the following Strichartz estimate:
\begin{lem}[Strichartz estimate]
Let $f \in L_x^2(\RR)$ and $k \in \mathbb{Z}$.  Then the Strichartz estimate
\begin{equation*}
\lVert e^{i t \Delta} f \rVert_{L_{t,x}^4} \lesssim \lVert f \rVert_{L_x^2}
\end{equation*}
holds, as does
the maximal function bound
\begin{equation*}
\lVert e^{i t \Delta} P_k f \rVert_{L_x^4 L_t^\infty} \lesssim
2^{k/2} \lVert f \rVert_{L_x^2}
\end{equation*}
\end{lem}
The first bound is the original Strichartz estimate \cite{St77}
and the second follows from scaling.

For a unit length $\theta \in \Sp^1$, we denote by $H_\theta$ its orthogonal
complement in $\R^2$ with the induced Lebesgue measure.
Define the lateral spaces $L_\theta^{p,q}$ as those consisting of all 
measurable $f$ for which the norm
\begin{equation*}
\lVert h \rVert_{L_\theta^{p,q}} :=
\left( \int_{\R} \left(
\int_{H_\theta \times \R} \lvert h(x_1 \theta + x_2, t )
\rvert^q dx_2 dt \right)^{p/q} dx_1 \right)^{1/p}
\end{equation*}
is finite.
We make the usual modifications when $p = \infty$ or $q = \infty$.
For proofs of the following lateral Strichartz estimates, see \cite[\S3, \S7]{BeIoKeTa11}.
\begin{lem}[Lateral Strichartz estimates]
Let $f \in L_x^2(\RR)$, $k \in \mathbb{Z}$, and $\theta \in \mathbb{S}^1$.
Let $ 2 < p \leq \infty, 2 \leq q \leq \infty$ and $1/p + 1/q = 1/2$. Then
\begin{align*}
\lVert e^{i t \Delta} P_{k, \theta} f \rVert_{L_\theta^{p, q}}
&\lesssim_{\phantom{p}}
2^{k(2/p - 1/2)} \lVert f \rVert_{L_x^2}, &p \geq q
\\
\lVert e ^{i t \Delta} P_k f \rVert_{L_\theta^{p,q}}
&\lesssim_p
2^{k(2/p - 1/2)} \lVert f \rVert_{L_x^2}, &p\leq q
\end{align*}
\label{Lateral}
\end{lem}
In the Schr\"odinger map setting, local smoothing spaces were first used in \cite{IoKe06}
and subsequently in \cite{IoKe07-2, BeIoKe07, Be08, BeIoKeTa11}.
\begin{lem}[Local smoothing \cite{IoKe06, IoKe07-2}]
Let $f \in L_x^2(\RR)$, $k \in \mathbb{Z}$, and $\theta \in \mathbb{S}^1$.  Then
\begin{equation*}
\lVert e^{i t \Delta} P_{k, \theta} f \rVert_{L_{\theta}^{\infty, 2}}
\lesssim
2^{-k/2} \lVert f \rVert_{L_x^2}
\end{equation*}
For $f \in L_x^2(\mathbb{R}^d)$, 
the maximal function space bound
\begin{equation*}
\lVert e^{i t \Delta} P_k f \rVert_{L_\theta^{2, \infty}} \lesssim
2^{k(d-1)/2} \lVert f \rVert_{L_x^2}
\end{equation*}
holds in dimension $d \geq 3$.
\label{L:LS}
\end{lem}
In $d = 2$, the maximal function bound fails due to a logarithmic divergence.
In order to overcome this, we exploit Galilean invariance as in \cite{BeIoKeTa11}
(the idea goes back to \cite{Tat01} in the setting of wave maps).

For $p, q \in [1, \infty]$, $\theta \in \mathbb{S}^1$, $\lambda \in \mathbb{R}$,
define $L_{\theta, \lambda}^{p,q}$ using the norm
\begin{equation*}
\lVert f \rVert_{L_{\theta, \lambda}^{p,q}} :=
\lVert T_{\lambda \theta}(f) \rVert_{L_\theta^{p,q}} =
\left( \int_{\R} \left(
\int_{H_\theta \times \R}
\lvert f((x_1 + t \lambda) \theta + x_2, t)\rvert^q dx_2 dt
\right)^{p/q} dx_1 \right)^{1/p}
\end{equation*}
where $T_w$ denotes the Galilean transformation
\begin{equation*}
T_w(f)(x,t) := e^{-i x \cdot w / 2} e^{-i t \lvert w \rvert^2 / 4} f(x + t w, t)
\end{equation*}
With $W \subset \mathbb{R}$ finite we define the spaces $L_{\theta, W}^{p,q}$ by
\begin{equation*}
L_{\theta, W}^{p,q} := \sum_{\lambda \in W} L_{\theta, \lambda}^{p,q}, \quad \quad
\lVert f \rVert_{L_{\theta, W}^{p,q}} :=
\inf_{f = \sum_{\lambda \in W} f_\lambda} \sum_{\lambda \in W}
\lVert f_\lambda \rVert_{L_{\theta, \lambda}^{p,q}}
\end{equation*}
For $k \in \mathbb{Z}$, $\mathcal{K} \in \mathbb{Z}_+$, set
\begin{equation*}
W_k := \{ \lambda \in [-2^k, 2^k] : 2^{k + 2 \mathcal{K}} \lambda \in \mathbb{Z} \}
\end{equation*}
In our application we shall work on a finite time interval
$[- 2^{2 \mathcal{K}}, 2^{2 \mathcal{K}}]$ in order to ensure that the $W_k$
are finite.  This still suffices for proving global results so long as our effective bounds
are proved with constants independent of $T, \mathcal{K}$.
As discussed in \cite[\S 3]{BeIoKeTa11}, restricting $T$ to a finite time interval
avoids introducing additional technicalities.

\begin{lem}[Local smoothing/maximal function estimates]
Let $f \in L_x^2(\RR)$, $k \in \mathbb{Z}$, and $\theta \in \mathbb{S}^1$.  Then
\begin{equation*}
\lVert e^{i t \Delta} P_{k, \theta} f \rVert_{L_{\theta, \lambda}^{\infty, 2}}
\lesssim 2^{-k / 2} \lVert f \rVert_{L_{x}^2}, \quad \quad
\lvert \lambda \rvert \leq 2^{k - 40}
\end{equation*}
and moreover, if $T \in (0, 2^{2 \mathcal{K}}]$, then 
\begin{equation*}
\lVert 1_{[-T, T]}(t) e^{i t \Delta} P_k f
\rVert_{L_{\theta, W_{k + 40}}^{2, \infty}} \lesssim
2^{k/2} \lVert f \rVert_{L_{x}^2}
\end{equation*}
\label{LSMF}
\end{lem}
\begin{proof}
The first bound follows from Lemma \ref{L:LS} via a Galilean boost.
The second is more involved and proven in \cite[\S 7]{BeIoKeTa11}.
\end{proof}

We now introduce the main function spaces, which follow the modifications
in \cite{Sm10} of spaces introduced in \cite{BeIoKeTa11}.

Let $T > 0$.  For $k \in \mathbb{Z}$, let $I_k = \{ \xi \in \mathbb{R}^2 : \lvert \xi \rvert
\in [2^{k-1}, 2^{k+1}] \}$.  
Let
\begin{equation*}
L_k^2(T) := \{ f \in L^2(\mathbb{R}^2 \times [-T, T] ) : \mathrm{supp}\; \hat{f}(\xi, t) \subset I_k \times
[-T, T] \}
\end{equation*}

For $f \in L^2 (\mathbb{R}^2 \times [-T, T] )$, let
\begin{equation*}
\begin{split}
\lVert f \rVert_{F_k^0(T)} :=& \;
\lVert f \rVert_{L_t^\infty L_x^2} +
\lVert f \rVert_{L_{t,x}^4}
+ 2^{-k/2} \lVert f \rVert_{L_x^4 L_t^\infty} \\
&+ 2^{-k/6} \sup_{\theta \in \mathbb{S}^1}
\lVert f \rVert_{L^{3,6}_\theta} + 
+ 2^{-k/2} \sup_{\theta \in \mathbb{S}^1}
\lVert f \rVert_{L_{\theta, W_{k + 40}}^{2, \infty}}
\end{split}
\end{equation*}
Define $F_k(T)$, $G_k(T)$, $N_k(T)$ as the normed spaces of functions in $L_k^2(T)$ for which the corresponding norms are finite:
\begin{align*}
\lVert f \rVert_{F_k(T)} :=& \inf_{J, m_1, \ldots, m_J \in \mathbb{Z}_+ }
\inf_{f = f_{m_1} + \cdots + f_{m_J}}
\sum_{j = 1}^J 2^{m_j} \lVert f_{m_j} \rVert_{F_{k + m_j}^0} \\
\lVert f \rVert_{G_k(T)} :=& \;\lVert f \rVert_{F_k^0(T)}
+ 2^{k / 6} \sup_{\lvert j - k \rvert \leq 20} \sup_{\theta \in \mathbb{S}^1}
\lVert P_{j, \theta} f \rVert_{L_\theta^{6,3}} \\
&+ 2^{k/2} \sup_{\lvert j - k \rvert \leq 20} \sup_{\theta \in \mathbb{S}^1}
\sup_{\lvert \lambda \rvert < 2^{k - 40}}
\lVert P_{j, \theta} f \rVert_{L_{\theta, \lambda}^{\infty, 2}} \\
\lVert f \rVert_{N_k(T)} :=& \inf_{f = f_1 + f_2 + f_3 + f_4 + f_5 + f_6}
\lVert f_1 \rVert_{L_{t,x}^{4/3}}
+ 2^{k/6} \lVert f_2 \rVert_{L_{\hat{\theta}_1}^{3/2, 6/5}}
+ 2^{k/6} \lVert f_3 \rVert_{L_{\hat{\theta}_2}^{3/2, 6/5}} \\
& + 2^{-k/6}
\lVert f_4 \rVert_{L^{6/5,3/2}_{\hat{\theta}_1}}
+ 2^{-k/6}
\lVert f_5 \rVert_{L^{6/5,3/2}_{\hat{\theta}_2}}
+ 2^{-k/2} \sup_{\theta \in \mathbb{S}^1} 
\lVert f_6 \rVert_{L_{\theta, W_{k - 40}}^{1,2}}
\end{align*}
where $(\hat{\theta}_1, \hat{\theta}_2)$ denotes the canonical basis in $\mathbb{R}^2$.

These spaces are related via the following linear estimate, which is proved in \cite{BeIoKeTa11}.
\begin{prop}[Main linear estimate]
Assume $\mathcal{K} \in \mathbb{Z}_+$, $T \in (0, 2^{2 \mathcal{K}}]$ and $k \in \mathbb{Z}$.  Then
for each $u_0 \in L^2$ that is frequency-localized to $I_k$ and for any $h \in N_k(T)$, the solution $u$ of
\begin{equation*}
(i \partial_t + \Delta_x) u = h, \quad\quad u(0) = u_0
\end{equation*}
satisfies
\begin{equation*}
\lVert u \rVert_{G_k(T)} \lesssim \lVert u(0) \rVert_{L_x^2} + \lVert h \rVert_{N_k(T)}
\end{equation*}
\label{MainLinearEstimate}
\end{prop}
The spaces $G_k(T)$ are used to hold projections $P_k \psi_m$ of the derivative fields
$\psi_m$ satisfying (\ref{NLS}).  The main components of $G_k(T)$ are the
local smoothing/maximal function spaces
$L_{\theta, \lambda}^{\infty, 2}$,
$L_{\theta, W_{k + 40}}^{2, \infty}$,
and the angular Strichartz spaces.  The local smoothing and
maximal function space components play 
an essential role in recovering the derivative loss that arises from the magnetic
nonlinearity.

The spaces $N_k(T)$ hold frequency projections of the nonlinearities in (\ref{NLS}).  
Here the main
spaces are the inhomogeneous local smoothing spaces $L_{\theta, W_{k - 40}}^{1,2}$
and the Strichartz spaces, both chosen to match those of $G_k(T)$.

The spaces $G_k(T)$ clearly embed in $F_k(T)$.  Two key properties enjoyed only by
the larger spaces $F_k(T)$ are
\begin{equation*}
\lVert f \rVert_{F_k(T)} \approx \lVert f \rVert_{F_{k+1}(T)}
\end{equation*}
for $k \in \mathbb{Z}$ and $f \in F_k(T) \cap F_{k+1}(T)$,
and
\begin{equation*}
\lVert P_k(uv) \rVert_{F_k(T)} \lesssim
\lVert u \rVert_{F_{k^\prime}(T)}
\lVert v \rVert_{L_{t,x}^\infty}
\end{equation*}
for $k, k^\prime \in \mathbb{Z}$, $\lvert k - k^\prime \rvert \leq 20$, $u \in F_{k^\prime}(T)$,
$v \in L^\infty(\RR \times [-T, T])$.
Both of these properties follow readily from the definitions.

In order to bound the nonlinearity of (\ref{NLS}) in $N_k(T)$, it is important to gain
regularity from the parabolic heat-time smoothing effect.  The desired frequency-localized
bounds do not (at least not readily) propagate in heat-time in the spaces $G_k(T)$, 
whereas these bounds do propagate with decay
in the larger spaces $F_k(T)$.
Note that since the $F_k(T)$ norm is translation
invariant, it holds that
\begin{equation*}
\lVert e^{s \Delta} h \rVert_{F_k(T)} \lesssim (1 + s 2^{2k})^{-20} \lVert h \rVert_{F_k(T)},
\quad \quad s \geq 0
\end{equation*}
for $h \in F_k(T)$.
In certain bilinear estimates we do not need the full strength of the spaces
$F_k(T)$ and instead can use the bound
\begin{equation}
\lVert f \rVert_{F_k(T)} \lesssim \lVert f \rVert_{L_x^2 L_t^\infty} + \lVert f \rVert_{L_{t,x}^4}
\label{FsoftBound}
\end{equation}
which follows from
\begin{equation*}
\lVert f \rVert_{L_{\theta, W_{k + m_j}}^{2, \infty}} \leq
\lVert f \rVert_{L_\theta^{2, \infty}}
\lesssim
2^{k/2} \lVert f \rVert_{L_x^2 L_t^\infty}
\end{equation*}

\subsection{Bilinear estimates}

For proofs of Lemmas \ref{L:NkBilinear}, \ref{L:L2Bilinear}, and \ref{L:L3Bilinear}, see
\cite[\S3]{BeIoKeTa11}.

\begin{lem}[Bilinear estimates on $N_k(T)$]
For $k, k_1, k_3 \in \mathbb{Z}$, $h \in L_{t,x}^2$, $f \in F_{k_1}(T)$, and $g \in G_{k_3}(T)$,
we have the following inequalities under the given restrictions on $k_1, k_3$.
\begin{align}
\lvert k_1 - k \rvert \leq 80: \quad
\lVert P_k (h f) \rVert_{N_k(T)} 
&\lesssim 
\lVert h \rVert_{L_{t,x}^2} \lVert f \rVert_{F_{k_1}(T)} 
\label{bilin1} \\
k_1 \leq k - 80: \quad
\lVert P_k (h f) \rVert_{N_k(T)} 
&\lesssim 2^{- \lvert k - k_1 \rvert /6} 
\lVert h \rVert_{L_{t,x}^2} \lVert f \rVert_{F_{k_1}(T)} 
\label{bilin2} \\
k \leq k_3 - 80: \quad
\lVert P_k (h g) \rVert_{N_k(T)} 
&\lesssim 2^{- \lvert k - k_3 \rvert /6} 
\lVert h \rVert_{L_{t,x}^2} \lVert g \rVert_{G_{k_3}(T)} 
\label{bilin3}
\end{align}
\label{L:NkBilinear}
\end{lem}

\begin{lem}[Bilinear estimates on $L_{t,x}^2$]
For $k_1, k_2, k_3 \in \mathbb{Z}$, $f_1 \in F_{k_1}(T)$, $f_2 \in F_{k_2}(T)$, and $g \in G_{k_3}(T)$,
we have
\begin{align}
\lVert f_1 \cdot f_2 \rVert_{L_{t,x}^2} &\lesssim \lVert f_1 \rVert_{F_{k_1}(T)} \lVert f_2 \rVert_{F_{k_2}(T)}
\label{bilin4} \\
k_1 \leq k_3: \quad
\lVert f \cdot g \rVert_{L_{t,x}^2} &\lesssim
2^{- \lvert k_1 - k_3 \rvert / 6} \lVert f \rVert_{F_{k_1}(T)} \lVert g \rVert_{G_{k_3}(T)}
\label{bilin5}
\end{align}
\label{L:L2Bilinear}
\end{lem}

We also have the following stronger estimates, which rely upon the local smoothing and
maximal function spaces.
\begin{lem}[Bilinear estimates using local smoothing/maximal function bounds]
For $k, k_1, k_2 \in \mathbb{Z}$, $h \in L^2_{t,x}$, $f \in F_{k_1}(T)$,
$g \in G_{k_2}(T)$, we have the following inequalities under
the given restrictions on $k_1, k_2$.
\begin{align}
k_1 \leq k - 80: \quad
&\lVert P_k (h f) \rVert_{N_k(T)} 
\lesssim 2^{- \lvert k - k_1 \rvert / 2} 
\lVert h \rVert_{L_{t,x}^2} \lVert f \rVert_{F_{k_1}(T)} 
\label{bilin2 improv}
\\
k_1 \leq k_2: \quad
&\lVert f \cdot g \rVert_{L_{t,x}^2} 
\lesssim
2^{- \lvert k_1 - k_2 \rvert / 2} \lVert f \rVert_{F_{k_1}(T)} 
\lVert g \rVert_{G_{k_2}(T)}
\label{bilin5 improv}
\end{align}
\label{L:L3Bilinear}
\end{lem}

\subsection{Trilinear Estimates and Summation}

We combine the bilinear estimates to
establish some trilinear estimates.
As we do not control local smoothing norms along the heat flow,
we will oftentimes be able to put only one term in a $G_k$
space. Nonetheless, such estimates still exhibit good off-diagonal
decay.

Define the sets $Z_1(k), Z_2(k), Z_3(k) \subset \mathbb{Z}^3$ as follows:
\begin{equation}
\begin{split}
Z_1(k) := 
&\{ (k_1, k_2, k_3) \in \mathbb{Z}^3 :
k_1, k_2 \leq k - 40 \; \mathrm{ and } \;  \lvert k_3 - k \rvert \leq 4 \} 
\\
Z_2(k) :=
&\{ (k_1, k_2, k_3) \in \mathbb{Z}^3 :
k, k_3 \leq k_1 - 40 \; \mathrm{ and } \;  \lvert k_2 - k_1 \rvert \leq 45\}
\\
Z_3(k) :=
&\{ (k_1, k_2, k_3) \in \mathbb{Z}^3 :
k_3 \leq k \; \mathrm{ and } \; \lvert k - \max\{k_1, k_2\} \rvert \leq 40 \\
& \; \mathbf{or} \;
k_3 > k \; \mathrm{and} \; \lvert k_3 - \max\{k_1, k_2\} \rvert \leq 40 \}
\end{split}
\label{Z def}
\end{equation}
In our main trilinear estimate, we avoid using local smoothing / maximal function
spaces. The following is proven in \cite[Lemma 3.10]{Sm10}.
\begin{lem}[Main trilinear estimate]
Let $C_{k,k_1,k_2,k_3}$ denote the best constant $C$ in the estimate
\begin{equation}
\lVert P_k \left( P_{k_1} f_1 P_{k_2} f_2 P_{k_3} g \right) \rVert_{N_k(T)} \lesssim
C
\lVert P_{k_1} f_1 \rVert_{F_{k_1}(T)}
\lVert P_{k_2} f_2 \rVert_{F_{k_2}(T)}
\lVert P_{k_3} g \rVert_{G_{k_3}(T)}
\label{trilinFFG}
\end{equation}
Then $C_{k,k_1,k_2,k_3}$ satisfies the bounds
\begin{displaymath}
C_{k,k_1,k_2,k_3} \lesssim
\left\{ \begin{array}{ll}
2^{ - \lvert (k_1 + k_2)/6 - k/3 \rvert} & (k_1, k_2, k_3) \in Z_1(k) \\
2^{- \lvert k - k_3 \rvert / 6} & (k_1, k_2, k_3) \in Z_2(k) \\
2^{- \lvert \Delta k \rvert / 6} & (k_1, k_2, k_3) \in Z_3(k) \\
0 & (k_1, k_2, k_3) \in \mathbb{Z}^3 \setminus \{ Z_1(k) \cup Z_2(k) \cup Z_3(k) \}
\end{array} \right.
\end{displaymath}
where $\Delta k = \max \{ k, k_1, k_2, k_3 \} - \min \{ k, k_1, k_2, k_3 \} \geq 0$.
\label{L:Trilin1}
\end{lem}
The following two corollaries are \cite[Corollary 3.11]{Sm10} and \cite[Corollary 3.12]{Sm10},
respectively.
\begin{cor}
Let $\{ a_k \}$, $\{ b_k \}$, $\{ c_k \}$ be $\delta$-frequency envelopes. 
Let $C_{k, k_1, k_2, k_3}$ be as in Lemma \ref{L:Trilin1}.
Then
\begin{equation*}
\sum_{(k_1, k_2, k_3) \in \mathbb{Z}^3 \setminus Z_2(k)} 
C_{k, k_1, k_2, k_3} a_{k_1} b_{k_2} c_{k_3}
\lesssim
a_k b_k c_k
\end{equation*}
\label{C:MTE1}
\end{cor}

\begin{cor}
Let $\{a_k \}, \{ b_k \}$ be $\delta$-frequency envelopes. 
Let $C_{k, k_1, k_2, k_3}$ be as in Lemma \ref{L:Trilin1}.
Then
\begin{equation*}
\sum_{(k_1, k_2, k_3) \in Z_2(k) \cup Z_3(k)}
2^{\max\{k, k_3\} - \max\{k_1, k_2\}}
C_{k, k_1, k_2, k_3} a_{k_1} b_{k_2} c_{k_3}
\lesssim a_k b_k c_k
\end{equation*}
\label{C:MTE2}
\end{cor}
If we take advantage of the local smoothing/maximal function spaces,
then we can obtain the following improvement (\cite[Lemma 3.13]{Sm10}).
\begin{lem}[Main trilinear estimate improvement over $Z_1$]
The best constant $C_{k,k_1,k_2,k_3}$ in $(\ref{trilinFFG})$
satisfies the improved estimate
\begin{equation}
C_{k,k_1,k_2,k_3}
\lesssim
2^{ - \lvert (k_1 + k_2)/2 - k \rvert}
\end{equation}
when $\{ k_1, k_2, k_3 \} \in Z_1(k)$.
\end{lem}
There are certain situations, such as when bounding the cubic-type term,
where we can place each term in a $G_k$ space, in which case we get
better estimates.
\begin{lem}[Improved trilinear estimate]
Let $C_{k,k_1,k_2,k_3}$ denote the best constant $C$ in the estimate
\begin{equation}
\lVert P_k \left( P_{k_1} g_1 P_{k_2} g_2 P_{k_3} g_3 \right) \rVert_{N_k(T)} \lesssim
C
\lVert P_{k_1} g_1 \rVert_{G_{k_1}(T)}
\lVert P_{k_2} g_2 \rVert_{G_{k_2}(T)}
\lVert P_{k_3} g_3 \rVert_{G_{k_3}(T)}
\label{trilinGGG}
\end{equation}
Then $C_{k,k_1,k_2,k_3}$ satisfies the bounds
\begin{displaymath}
C_{k,k_1,k_2,k_3} \lesssim
\left\{ \begin{array}{ll}
2^{ - \lvert (k_1 + k_2)/2 - k \rvert} & (k_1, k_2, k_3) \in Z_1(k) \\
2^{- \lvert k - k_1 \rvert / 6} 
2^{- \lvert k_3 - k_2 \rvert / 6} & (k_1, k_2, k_3) \in Z_2(k) \\
2^{- \lvert \Delta k \rvert / 6} & (k_1, k_2, k_3) \in Z_3(k) \\
0 & (k_1, k_2, k_3) \in \mathbb{Z}^3 \setminus \{ Z_1(k) \cup Z_2(k) \cup Z_3(k) \}
\end{array} \right.
\end{displaymath}
where $\Delta k = \max \{ k, k_1, k_2, k_3 \} - \min \{ k, k_1, k_2, k_3 \} \geq 0$.
\label{L:GGGTrilin}
\end{lem}
\begin{proof}
We seek an improvement over Lemma \ref{L:Trilin1} only on the set $Z_2(k)$.
Here we apply (\ref{bilin3}) so that
\begin{equation*}
\lVert P_k(P_{k_1} g_1 P_{k_2} g_2 P_{k_3} g_3) \rVert_{N_k(T)}
\lesssim
2^{-\lvert k - k_1 \rvert/6} \lVert P_{k_1} g_1 P_{k_2} g_2 \rVert_{L^2_{t,x}}
\lVert P_{k_3} g_3 \rVert_{G_{k_3}(T)}
\end{equation*}
We conclude with (\ref{bilin5}).
\end{proof}

\begin{cor}
Let $\{a_k\}$, $\{b_k\}$, $\{c_k\}$ be $\delta$-frequency envelopes. Let
$C_{k,k_1,k_2,k_3}$ be as in Lemma \ref{L:GGGTrilin}. Then
\begin{equation*}
\sum_{(k_1, k_2, k_3) \in \mathbb{Z}^3}
C_{k,k_1,k_2,k_3} \lesssim a_k b_k c_k
\end{equation*}
\label{C:GGG}
\end{cor}
\begin{proof}
In view of Corollary (\ref{C:MTE1}) we need only establish the bound on
$Z_2(k)$. We have
\begin{align*}
&\sum_{(k_1, k_2, k_3) \in Z_2(k)} 
2^{- \lvert k - k_1 \rvert / 6} 2^{- \lvert k_3 - k_2 \rvert / 6}
a_{k_1} b_{k_2} c_{k_3} \\
&\lesssim
\sum_{k, k_3 \leq k_1 - 40} 
2^{- \lvert k - k_1 \rvert / 6} 2^{- \lvert k_3 - k_1 \rvert / 6}
a_{k_1} b_{k_1} c_{k_3} \\
&\lesssim
a_k b_k c_k
\sum_{k, k_3 \leq k_1 - 40} 
2^{- \lvert k - k_1 \rvert / 6} 2^{- \lvert k_3 - k_1 \rvert / 6}
2^{3 \delta \lvert k - k_1 \rvert} 2^{\delta \lvert k_3 - k_1 \rvert}
\end{align*}
To finish, sum on $k_3$, then on $k_1$.
\end{proof}

\section{Bounds along the heat flow}

We recall from \cite{BeIoKeTa11, Sm10} 
some estimates in the $F_k$ spaces that propagate along the heat flow.
We assume throughout that $\phi$ has subthreshold energy and
that $\varepsilon > 0$ is a very small number such that
$b_k \leq \varepsilon$ and
$\varepsilon^{1/2} \sum_j b_j^2 \ll 1.$
\begin{lem}
Let $k \in \mathbb{Z}$, $s \geq 0$.
Let
$F \in \{ A_\ell^2, \partial_\ell A_\ell, fg : \ell = 1, 2;
f, g \in \{ \psi_m, \overline{\psi_m}: m = 1,2\} \}$.
Then
\begin{align}
\lVert P_k \psi_m(s) \rVert_{F_k(T)}
&\lesssim
(1 + s 2^{2k})^{-4} 2^{-\sigma k} b_k(\sigma) 
\label{SSV1} \\
\lVert P_k (A_\ell \psi_m(s)) \rVert_{F_k(T)}
&\lesssim
\varepsilon (1 + s 2^{2k})^{-3} (s 2^{2k})^{-3/8} 2^k 2^{-\sigma k} b_k(\sigma)
\label{PAPsi FS}  \\
\lVert P_k F(s) \rVert_{F_k(T)} 
&\lesssim
\varepsilon^{1/2}
(1 + s 2^{2k})^{-2} (s 2^{2k})^{-5/8} 
2^k 2^{-\sigma k} b_k(\sigma) \label{Quadratic Parabolic} \\
\lVert P_k \int_0^s e^{(s - s^\prime) \Delta} U_m(s^\prime) ds^\prime \rVert_{F_k(T)}
&\lesssim \varepsilon (1 + s 2^{2k})^{-4}2^{-\sigma k} b_k(\sigma)
\label{UFk}
\\
\lVert P_k A_m(0) \rVert_{F_k(T)}
&\lesssim
\sum_p b_p^2
\label{PA Sk}
\end{align}
\end{lem}
These $F_k$ estimates can then be used to establish various $L^4$ bounds
along the heat flow. In particular, we have \cite[Lemma 7.17]{Sm10}:
\begin{lem}
The following bounds hold
\begin{align}
\lVert P_k \psi_t(s) \rVert_{L^4_{t,x}} +
\lVert P_k \psi_s(s) \rVert_{L_{t,x}^4}
&\lesssim
(1 + s 2^{2k})^{-2} 2^k
(1 + \sum_j b_j^2) b_k
\label{C:Psist L4}
\\
\lVert \int_0^s e^{(s-r)\Delta} P_k U_{\alpha}(r) dr \rVert_{L^4_{t,x}}
&\lesssim
\varepsilon
(1 + s 2^{2k})^{-2} 2^k
(1 + \sum_j b_j^2) b_k
\label{L:L4NonlinearPart}
\end{align}
where $\alpha \in \{0, 3\}$.
\end{lem}
Taking advantage of the fact that the heat and Schr\"odinger flows
share the same initial data at $s = 0$ for all $t$, we may use the above
bounds to control the commutator of these flows with an $\varepsilon$ gain.
\begin{cor}
Let $\Psi = \psi_t - i \psi_s$. Then
\begin{equation*}
\lVert P_k \Psi \rVert_{L^4_{t,x}}
\lesssim \varepsilon 2^k (1 + \sum_j b_j^2) b_k
\end{equation*}
\label{C:Comm L4}
\end{cor}
\begin{proof}
From (\ref{genPsiEQ}) and (\ref{Heat Nonlinearity 0})
we have
\begin{equation*}
(\partial_s - \Delta) \Psi = U_t - i U_s
\end{equation*}
As $\Psi(s = 0) = 0$, Duhamel implies
\begin{equation*}
\Psi(s) = \int_0^s e^{(s - r)\Delta}( U_t - i U_s )(r) dr
\end{equation*}
The conclusion follows from \eqref{L:L4NonlinearPart}.
\end{proof}

This next lemma is a special case of \cite[Lemma 7.18]{Sm10}:
\begin{lem}
It holds that
\begin{equation}
\lVert P_k A_m(0) \rVert_{L_{t,x}^4} \lesssim
b_k^2
\label{Pk A4}
\end{equation}
\label{L:Pk A4}
\end{lem}

\section{The electric potential}
\label{SS:Electric}

In this section we prove that $V_m$ (defined by \ref{V Def}) is perturbative
in the sense that $\lVert P_k V_m \rVert_{N_k(T)} \lesssim \varepsilon b_k$.
Throughout this section, $\varepsilon > 0$ is a very small number such that
$b_k \leq \varepsilon$ and
$\varepsilon^{1/2} \sum_j b_j^2 \ll 1.$ In application, it will be set equal to a suitable power
of the energy dispersion parameter $\varepsilon_0$.
\begin{prop}
The term $V_m = (A_t + A_x^2) \psi_m - i \psi_\ell \Im(\bar{\psi}_\ell \psi_m)$
is perturbative.
\end{prop}
We prove this in several steps, starting with the cubic term.

\textbf{The cubic term $- i \psi_\ell \Im(\bar{\psi}_\ell \psi_m)$.}

To control the cubic term we use off-diagonal decay and take advantage of the fact that any $\psi$
may be placed in a $G_k$ space. Hence Lemma \ref{L:GGGTrilin} applies
and from Corollary \ref{C:GGG} we conclude
\begin{equation*}
\lVert P_k ( \psi_\ell \Im(\bar{\psi}_\ell \psi_m) )\rVert_{N_k(T)}
\lesssim \varepsilon b_k
\end{equation*}

\textbf{The term $A_x^2 \psi$.}

We conclude as a corollary of Lemma \ref{L:Pk A4} that
\begin{equation}
\lVert A_x^2(0) \rVert_{L^2_{t,x}}
\lesssim
\lVert A_x(0) \rVert_{L^4_{t,x}}^2
\lesssim
\sum_{k \in \mathbb{Z}} \lVert P_k A_x(0) \rVert_{L^4_{t,x}}^2
\lesssim
\sup_{j \in \mathbb{Z}} b_j^2 \cdot \sum_{k \in \mathbb{Z}} b_k^2
\label{Ax2 L2}
\end{equation}
We next show how to use this $L^2$ bound to control $A_x^2 \psi_m$
in $N_k$ spaces.
\begin{lem}
Let $f \in L^2_{t,x}$. Then
\begin{equation}
\lVert P_k (f \psi_m) \rVert_{N_k(T)} \lesssim
\lVert f \rVert_{L^2_{t,x}} b_k
\end{equation}
\label{L:VmBound}
\end{lem}
\begin{proof}
Begin with the following Littlewood-Paley decomposition of $P_k(f \psi_x)$:
\begin{align*}
P_k(f \psi_x) =& \; P_k (P_{<k - 80}f P_{k - 5 < \cdot < k + 5} \psi_x) \\
& \;+
\sum_{ \substack{\lvert k_1 - k \rvert \leq 4 \\ k_2 \leq k - 80}} P_k(P_{k_1}f P_{k_2} \psi_x) +
\sum_{ \substack{\lvert k_1 - k_2 \rvert \leq 90 \\ k_1, k_2 > k - 80}}
P_k( P_{k_1}f P_{k_2} \psi_x )
\end{align*}
The first term is controlled using (\ref{bilin1}):
\begin{align*}
\lVert P_k (P_{<k - 80}f P_{k - 5 < \cdot < k + 5} \psi_x) \rVert_{N_k(T)}
&\leq \lVert P_k (P_{<k - 80}f P_{k - 5 < \cdot < k + 5} \psi_x) \rVert_{L_{t,x}^{4/3}} \\
&\leq \lVert P_{< k - 80} f \rVert_{L_{t,x}^2} \lVert P_{k - 5 < \cdot < k + 5} \psi_x \rVert_{G_k(T)}
\end{align*}
To control the second term we apply (\ref{bilin2}):
\begin{equation*}
\lVert P_k (P_{k_1} f P_{k_2} \psi_x) \rVert_{N_k(T)} \lesssim 2^{- \lvert k_2 - k \rvert / 6}
\lVert P_{k_1} f \rVert_{L_{t,x}^2} \lVert P_{k_2} \psi_x \rVert_{G_{k_2}(T)}
\end{equation*}
To control the high-high interaction, apply (\ref{bilin3}):
\begin{equation*}
\lVert P_k( P_{k_1}f P_{k_2} \psi_x ) \rVert_{N_k(T)}
\lesssim
2^{- \lvert k - k_2 \rvert / 6} \lVert P_{k_1} f \rVert_{L_{t,x}^2} \lVert P_{k_2} \psi_x \rVert_{G_{k_2}(T)}
\end{equation*}
Therefore
\begin{equation*}
\sum_{ \substack{\lvert k_1 - k_2 \rvert \leq 90 \\ k_1, k_2 > k - 80}}
\lVert 
P_k( P_{k_1}f P_{k_2} \psi_x )
\rVert_{N_k(T)}
\lesssim
\sum_{ \substack{\lvert k_1 - k_2 \rvert \leq 90 \\ k_1, k_2 > k - 80}}
2^{-\lvert k - k_2\rvert/6} \lVert P_{k_1} f \rVert_{L_{t,x}^2} b_{k_2}
\end{equation*}
and so by Cauchy-Schwarz
\begin{equation*}
\sum_{ \substack{\lvert k_1 - k_2 \rvert \leq 90 \\ k_1, k_2 > k - 80}}
\lVert 
P_k( P_{k_1}f P_{k_2} \psi_x )
\rVert_{N_k(T)}
\lesssim
b_k
\left( \sum_{k_1 \geq k - 80} \lVert P_{k_1} f \rVert_{L_{t,x}^2}^2 \right)^{1/2}
\end{equation*}
Upon interchanging the $L_{t,x}^2$ and $\ell^2$ norms, 
we conclude from the standard square function estimate that
\begin{equation*}
\sum_{ \substack{\lvert k_1 - k_2 \rvert \leq 90 \\ k_1, k_2 > k - 80}}
\lVert 
P_k( P_{k_1}f P_{k_2} \psi_x )
\rVert_{N_k(T)}
\lesssim
\lVert f \rVert_{L_{t,x}^2} b_k
\end{equation*}
\end{proof}
Together (\ref{Ax2 L2}) and Lemma \ref{L:VmBound}
imply
\begin{equation}
\lVert P_k(A_x^2 \psi_m) \rVert_{N_k(T)}
\lesssim
\varepsilon b_k
\end{equation}

\textbf{The leading term $A_t \psi$.}

This term requires more effort to bound, as its behavior blends
that of the cubic term and $A_x^2 \psi$.
The main difficulty arises from the fact that we do not control
$\psi_t(s)$ in $F_k$ spaces for positive heat flow times $s > 0$.
While at $s = 0$ we do indeed have $\psi_t = i D_j \psi_j$ 
(because at $s = 0$ it holds that $\psi_t = i \psi_s$)
as a consequence
of the fact that $\phi$ is a Schr\"odinger map, along the heat
flow we do not have such an explicit representation of $\psi_t$
and instead must access it through the commutator of the heat
and Schr\"odinger flows.
Thus our first step is to represent $\psi_t$ as $i \psi_s + \Psi$.
It may be tempting to try to place $\Psi$ in the $F_k$ spaces,
as we do have such bounds for $\psi_s$. The $F_k$ bounds
for $\psi_s$, however, are obtained as a consequence of the formula
$\psi_s = D_j \psi_j$ and
the bounds on $\psi_x$ and $A_x$; therefore
a different approach would
be required to bound $\Psi$ in $F_k$.
The bounds on $\Psi(s), s > 0$ that can be readily obtained
are those in $L^4$. 
Owing to the fact that $\Psi(0) = 0$,
it turns out that under the assumption of energy dispersion, these bounds
come with an $\varepsilon$ gain.
Representing $A_t$ as
\begin{equation}
A_t(s) = - \int_0^\infty \Im(i \bar{\psi}_s \psi_s)(s^\prime) ds^\prime
- \int_0^\infty \Im(\bar{\Psi} \psi_s)(s^\prime) ds^\prime
\label{At sep}
\end{equation}
we show that the $L^2$ norm of the second integral is small,
enabling us to treat its total contribution to $A_t \psi$ using
Lemma \ref{L:VmBound}.

Now $\psi_s = D_j \psi_j$ and as already mentioned does 
enjoy bounds in the $F_k$-spaces.
The effect of the integration is cancel the derivatives that appear, and so
in principle the first integral in (\ref{At sep}) is on par with the cubic term
$\psi_\ell \Im(\bar{\psi}_\ell \psi_m)$.
However, any one of the terms in the cubic term may be placed in a $G_k$
space, whereas, for $s^\prime > 0$, we
cannot place $\psi_s(s^\prime)$ in $G_k$ spaces.
For this reason we further decompose $\psi_s$, representing it as
$\psi_s = \partial_\ell \psi + i A_\ell \psi_\ell$.
Rewriting the first integral in (\ref{At sep}) as
\begin{equation*}
- \int_0^\infty \lvert \psi_s \rvert^2 (s^\prime) ds^\prime
\end{equation*}
we see that
it suffices (by Young's inequality) to just consider the contributions from
$\lvert \partial_\ell \psi \rvert^2$ and from $\lvert A_\ell \psi_\ell \rvert^2$.
For the latter term, we pull out $A_\ell$ in $L^2$ (which is largest at $s = 0$), and control
$\int_0^\infty \lvert \psi_\ell \rvert^2 \lesssim_{E_0} 1$ using heat flow bounds; hence
this term has a contribution like that of $A_x^2 \psi$. 
For the $\lvert \partial_\ell \psi \rvert^2$ term,
we expand $\psi(s)$ using the Duhamel formula
as the sum of its linear evolution and nonlinear evolution. 
While we do not propagate bounds in the $G_k$-spaces along the nonlinear
heat flow, such bounds do in fact propagate along the linear flow;
the contribution from the three linear terms taken together
therefore is comparable to that of the cubic term.
The nonlinear evolution terms must be dealt with more delicately,
as here we must resort to placing these in $F_k$-spaces, resulting
in worse off-diagonal gains. The upshot, however, is that these terms come
with an energy-dispersion gain that is enough to offset the consequences of inferior decay.
We turn to the details.
\begin{lem}
It holds that
\begin{equation*}
\lVert
\int_0^\infty  (\overline{\Psi} \cdot D_\ell \psi_\ell)(s)\ds
 \rVert_{L_{t,x}^2} 
\lesssim
\varepsilon (1 + \sum_j b_j^2)^2 \sum_k b_k^2
\end{equation*}
\label{L:AtL4}
\end{lem}
\begin{proof}
We first bound
$(\overline{\Psi} \cdot D_\ell \psi_\ell)(s)$ in $L^2$. 
Define
\begin{equation*}
\mu_k(s) := \sup_{k^\prime \in \mathbb{Z}} 2^{-\delta \lvert k - k^\prime \rvert}
\lVert P_k \Psi(s) \rVert_{L_{t,x}^4}
\quad \text{and} \quad
\nu_k(s) := \sup_{k^\prime \in \mathbb{Z}} 2^{-\delta \lvert k - k^\prime \rvert}
\lVert P_k (D_\ell \psi_\ell)(s) \rVert_{L_{t,x}^4}
\end{equation*}
Then
\begin{equation}
\lVert (\overline{\Psi} \cdot D_\ell \psi_\ell)(s) \rVert_{L_{t,x}^2} 
\lesssim
\sum_k \mu_k(s) \sum_{j \leq k} \nu_j(s) +
\sum_k \nu_k(s) \sum_{j \leq k} \mu_j(s)
\label{L2  Envelope Bound}
\end{equation}
From Corollary \ref{C:Comm L4} and from \eqref{C:Psist L4} it follows that
\begin{equation}
\mu_k(s) \lesssim \varepsilon (1 + s 2^{2k})^{-2} 2^k (1 + \sum_p b_p^2) b_k
\label{mu Bound}
\end{equation}
and
\begin{equation}
 \nu_k(s) \lesssim (1 + s 2^{2k})^{-2} 2^k (1 + \sum_p b_p^2) b_k
\label{nu Bound}
\end{equation}
Therefore
\begin{align*}
\int_0^\infty \lVert (\overline{\Psi} \cdot D_\ell \psi_\ell)(s) \rVert_{L_{t,x}^2} \ds
\lesssim& \;
\int_0^\infty \sum_k \mu_k(s) \sum_{j \leq k} \nu_j(s) ds \\
\lesssim&\; \varepsilon
(1 + \sum_p b_p^2)^2 \sum_k 2^k b_k
\sum_{j \leq k} 2^j b_j  \times \\
& \; \times
\int_0^\infty (1 + s 2^{2j})^{-2} (1 + s 2^{2k})^{-2} \ds
\end{align*}
Then
\begin{equation*}
\sum_{j \leq k} 2^j b_j
\int_0^\infty (1 + s 2^{2j})^{-2} (1 + s 2^{2k})^{-2} \ds
\lesssim
\int_0^\infty (1 + s 2^{2k})^{-2} \ds \sum_{j \leq k} 2^j b_j
\end{equation*}
and
\begin{equation*}
\int_0^\infty (1 + s 2^{2k})^{-2} \ds \lesssim 2^{-2k}
\end{equation*}
so that
\begin{equation*}
\sum_k 2^k b_k
\sum_{j \leq k} 2^j b_j
\int_0^\infty (1 + s 2^{2j})^{-2} (1 + s 2^{2k})^{-2} \ds
\lesssim
\sum_k b_k^2
\end{equation*}
\end{proof}
\begin{lem}
It holds that
\begin{equation*}
\lVert
\int_0^\infty \lvert A_\ell \psi_\ell \rvert^2(s) ds
\rVert_{L^2_{t,x}}
\lesssim
\sup_{j} b_j^2 \sum_{k} b_k^2
\end{equation*}
\end{lem}
\begin{proof}
Start by taking a Littlewood-Paley decomposition
of $\lvert A_\ell \rvert^2$:
\begin{equation*}
\int_0^\infty \lvert A_\ell \psi_\ell \rvert^2(s) ds
=
\int_0^\infty 
\sum_{j, k}
P_j A_\ell(s) \cdot  P_k \overline{A}_\ell(s) \cdot
\lvert \psi_x \rvert^2(s) ds
\end{equation*}
We bound this expression in absolute value by
\begin{equation*}
\sum_{j, k} \sup_{s \geq 0}
\lvert P_j A_x (s) \rvert \cdot \sup_{s^\prime \geq 0}
\lvert P_k A_x (s^\prime) \rvert
\cdot
\int_0^\infty \lvert \psi_x \rvert^2(r) dr
\end{equation*}
Next we take the $L^2_{t,x}$ norm, which we control
by placing the integral in $L^\infty_{t,x}$ and the
summation in $L^2_{t,x}$.
We control the summation by
\begin{equation*}
\sum_{j, k} 
\lVert 
\sup_{s \geq 0}
\lvert P_j A_x (s) \rvert \cdot \sup_{s^\prime \geq 0}
\lvert P_k A_x (s^\prime) \rvert
\rVert_{L^2_{t,x}}
\lesssim
\sum_j \lVert \sup_{s \geq 0} \lvert P_j A_x(s) \rvert \rVert_{L^4_{t,x}}^2
\end{equation*}
Now
\begin{equation*}
\sup_{s \geq 0} \lvert P_j A_m(s) \rvert
\lesssim
\sup_{s \geq 0}
\int_s^\infty \lvert P_j (\overline{\psi_m(s^\prime)}
D_\ell \psi_\ell(s^\prime)) \rvert ds^\prime
\leq
\int_0^\infty
\lvert P_j (\overline{\psi_m(s^\prime)}
D_\ell \psi_\ell(s^\prime)) \rvert ds^\prime
\end{equation*}
and in view of the proof of Lemma \ref{L:Pk A4},
the $L^2$ norm of the right hand side
is bounded by $b_j^2$.

Thus it remains to show
\begin{equation*}
\lVert 
\int_0^\infty \lvert \psi_\ell \rvert^2(s) ds
\rVert_{L^\infty_{t,x}}
\lesssim 1
\end{equation*}
For fixed $s, x, t$, however, $\psi_\ell$ is simply
the representation of $\partial_\ell \phi$ with
respect to an orthonormal basis of
$T_{\phi(s, x, t)} \mathbb{S}^2$. Therefore
$\lvert \psi_\ell \rvert = \lvert \partial_\ell \phi \rvert$
and so we may invoke (the uniform in time)
estimate (\ref{ek4}).
\end{proof}
\begin{lem}
It holds that
\begin{equation}
P_k\left( \int_0^\infty \lvert \partial_\ell \psi_\ell \rvert^2(s^\prime) ds^\prime
\psi_m(0) \right) \lesssim \varepsilon b_k
\label{dpdpp}
\end{equation}
\end{lem}
\begin{proof}
We write
\begin{equation*}
\psi_\ell(s) = e^{s \Delta} \psi_\ell(0) + \int_0^s e^{(s - r)\Delta} U_\ell(r) dr
\end{equation*}
and expand (\ref{dpdpp}) accordingly.
Additionally, we perform a Littlewood-Paley decomposition of each term.
The trilinear term
\begin{equation*}
P_k \left( \sum_{k_1, k_2, k_3}
\int_0^\infty e^{s \Delta}P_{k_1} \partial_\ell \psi_\ell(0)
\cdot
e^{s \Delta} P_{k_2} \partial_\ell \psi_\ell(0) \;ds \cdot P_{k_3} \psi_m(0) \right)
\end{equation*}
we bound in $N_k(T)$ by
\begin{equation*}
\begin{split}
\sum_{\substack{ (k_1, k_2, k_3) \in \\ Z_1(k) \cup Z_2(k) \cup Z_3(k)}}
& C_{k,k_1,k_2,k_3} 2^{k_1} \lVert P_{k_1} \psi_x(0) \rVert_{G_{k_1}(T)}
2^{k_2} \lVert P_{k_2} \psi_x(0) \rVert_{G_{k_2}(T)} \times \\
& \times \lVert P_{k_3} \psi_m(0) \rVert_{G_{k_3}(T)} 
\int_0^\infty (1 + s 2^{2k_1})^{-20} (1 + s 2^{2k_2})^{-20} ds
\end{split}
\end{equation*}
where $C_{k,k_1,k_2,k_3}$ is as in Lemma \ref{L:GGGTrilin}
and $Z_1, Z_2, Z_3$ are given by (\ref{Z def}).
As
\begin{equation*}
2^{k_1} 2^{k_2} \int_0^\infty (1 + s 2^{2k_1})^{-20} (1 + s 2^{2k_2})^{-20} ds
\lesssim 1
\end{equation*}
it suffices to control
\begin{equation}
\begin{split}
\sum_{\substack{ (k_1, k_2, k_3) \in \\ Z_1(k) \cup Z_2(k) \cup Z_3(k)}}
&C_{k,k_1,k_2,k_3}
\lVert P_{k_1} \psi_x(0) \rVert_{G_{k_1}(T)} \times \\
&\times \lVert P_{k_2} \psi_x(0) \rVert_{G_{k_2}(T)}
\lVert P_{k_3} \psi_m(0) \rVert_{G_{k_3}(T)}
\end{split}
\label{ppp}
\end{equation}
We invoke Corollary \ref{C:GGG} to bound (\ref{ppp}) by $b_k^3$,
which suffices for this term in view of the energy dispersion assumption.

Next we must consider products involving either one or two
$\int_0^s e^{(s - r)\Delta} U_\ell(r) dr$ terms.
The arguments are similar to those of the case already considered.
Here, however, we only place
$P_{k_3} \psi_m(0)$ in a $G_k$ space; the remaining two
terms we control in $F_k$ spaces using either
(\ref{SSV1}) or (\ref{UFk}).
Hence in our $N_k(T)$ bound we use $C_{k,k_1,k_2,k_3}$ as in
Lemma \ref{L:Trilin1} rather than Lemma \ref{L:GGGTrilin},
and to sum we must use Corollary \ref{C:MTE1} instead of
Corollary \ref{C:GGG}.
Corollary \ref{C:MTE1}, however, only supplies the bound
over $\mathbb{Z}^3 \setminus Z_2(k)$.
On $Z_2$ the sum is controlled by
\begin{align*}
\varepsilon \sum_{(k_1, k_2, k_3) \in Z_2(k)} 2^{- \lvert k - k_3 \rvert / 6} b_{k_1} b_{k_2} b_{k_3}
&\lesssim
\varepsilon  \sum_{k, k_3 \leq k_1 - 40} 2^{- \lvert k - k_3 \rvert / 6} b_{k_1} b_{k_1} b_{k_3} \\
&\lesssim
\varepsilon  b_k
\sum_{k, k_3 \leq k_1 - 40}  2^{-\lvert k - k_3 \rvert / 6} 2^{\delta \lvert k - k_3 \rvert} b_{k_1}^2 \\
&\lesssim
\varepsilon b_k \sum_{k_1 \geq k + 40} b_{k_1}^2
\end{align*}
\end{proof}

\section{The magnetic potential}
\label{SS:DMP}

We begin by introducing a paradifferential decomposition of the magnetic nonlinearity,
splitting it into two pieces. 
This decomposition depends upon a frequency
parameter $k \in \mathbb{Z}$, which we suppress in the notation;
this frequency $k$ is the output frequency.
The decomposition also depends upon a universal frequency gap parameter $\varpi \in \mathbb{Z}_+$
that need only be taken sufficiently large (see \cite[\S 5]{Sm10} for discussion).

Define
\[
\begin{split}
A_{m, \mathrm{lo \wedge lo}}(s)
&:=
-\sum_{k_1, k_2 \leq k - \varpi}
\int_s^\infty \Im( \overline{P_{k_1} \psi_m} P_{k_2} \psi_s )(s^\prime) ds^\prime 
\\
A_{m, \mathrm{hi \lor hi}}(s)
&:=
-\sum_{\max\{k_1, k_2\} > k - \varpi}
\int_s^\infty \Im( \overline{P_{k_1} \psi_m} P_{k_2} \psi_s )(s^\prime) ds^\prime
\end{split}
\]
so that
$A_m = A_{m, \mathrm{lo \wedge lo}} + A_{m, \mathrm{hi \lor hi}}$.
Similarly define
\[
\begin{split}
B_{m, \mathrm{lo \wedge lo}} 
&:= - i \sum_{k_3} \left( \partial_\ell( A_{\ell, \mathrm{lo \wedge lo}} P_{k_3} \psi_m)
+
A_{\ell, \mathrm{lo \wedge lo}} \partial_\ell P_{k_3} \psi_m \right)
\\
B_{m, \mathrm{hi \lor hi}} 
&:= - i \sum_{k_3} \left( \partial_\ell( A_{\ell, \mathrm{hi \lor hi}} P_{k_3} \psi_m)
+
A_{\ell, \mathrm{hi \lor hi}} \partial_\ell P_{k_3} \psi_m \right)
\end{split}
\]
so that
$B_m = B_{m, \mathrm{lo \wedge lo}} + B_{m, \mathrm{hi \lor hi}}$.

Our goal is to control $P_k B_m$ in $N_k(T)$.
We consider first $P_k B_{m, \mathrm{hi \lor hi}}$,
performing a Littlewood-Paley decomposition.
In order for frequencies $k_1, k_2, k_3$ to have an output in
this expression at a frequency $k$, we must have
$(k_1, k_2, k_3) \in Z_2(k) \cup Z_3(k) \cup Z_0(k)$,
where $Z_0(k) := Z_1(k) \cap \{ (k_1, k_2, k_3) \in \mathbb{Z}^3 :
k_1, k_2 > k - \varpi \}$ and $Z_2, Z_3$ are given by (\ref{Z def}).
We apply Lemma \ref{L:Trilin1} to bound
$P_k B_{m, \mathrm{hi \lor hi}}$ in $N_k(T)$ by
\begin{equation*}
\begin{split}
\sum_{\substack{ (k_1,k_2,k_3) \in \\ Z_2(k) \cup Z_3(k)
\cup Z_0(k)}}
& \int_0^\infty
2^{\max\{k, k_3\}} C_{k,k_1,k_2,k_3} \lVert P_{k_1} \psi_x(s) \rVert_{F_{k_1}} \times \\
& \times \lVert P_{k_2} (D_\ell \psi_\ell(s)) \rVert_{F_{k_2}}
 \lVert P_{k_3} \psi_m(0) \rVert_{G_{k_3}} ds
\end{split}
\end{equation*}
which, thanks to (\ref{SSV1}) and (\ref{PAPsi FS}), is
controlled by
\begin{equation*}
\begin{split}
\sum_{\substack{ (k_1,k_2,k_3) \in \\ Z_2(k) \cup Z_3(k)
\cup Z_0(k)}}
&2^{\max\{k, k_3\}} C_{k,k_1,k_2,k_3} b_{k_1} b_{k_2} b_{k_3} \times \\
& \times \int_0^\infty
(1 + s2^{2k_1})^{-4} 2^{k_2} (s 2^{2k_2})^{-3/8}
(1 + s2^{2k_2})^{-2}
ds
\end{split}
\end{equation*}
As
\begin{equation}
\int_0^\infty
(1 + s2^{2k_1})^{-4} 2^{k_2} (s 2^{2k_2})^{-3/8}
(1 + s2^{2k_2})^{-2}
ds
\lesssim
2^{-\max \{k_1, k_2\}}
\label{Ik1k2}
\end{equation}
we reduce to
\begin{equation*}
\sum_{\substack{ (k_1,k_2,k_3) \in \\ Z_2(k) \cup Z_3(k)
\cup Z_0(k)}}
2^{\max\{k, k_3\} - \max\{k_1,k_2\}} 
C_{k,k_1,k_2,k_3}
b_{k_1} b_{k_2} b_{k_3}
\end{equation*}

To estimate $P_k B_{m, \mathrm{hi \lor hi}}$ on $Z_2 \cup Z_3$,
we apply Corollary \ref{C:MTE2} and use the energy dispersion
hypothesis.
As for $Z_0(k)$, we note that its cardinality $\lvert Z_0(k) \rvert$
satisfies $\lvert Z_0(k) \rvert \lesssim \varpi$ independently
of $k$. Hence for fixed $\varpi$ summing over this set
is harmless given sufficient energy dispersion.

The last term to consider is $P_k B_{m, \mathrm{lo \wedge lo}}$.
We cannot control this term directly in $N_k$, but instead can
get a handle on it via the following bilinear estimate from
\cite[\S 5]{Sm10}.
\begin{thm}
Let $\sigma \in [0, \sigma_1 - 1]$, $s \geq 0$, and $2^{j - k} \ll1$.  
Suppose that $V_m$ and $B_{m, \mathrm{hi \lor hi}}$
are perturbative and that this can be established
using only mixed $L^p$ spaces with all $p$ lying in a compact
subset of $(1, \infty)$, e.g., the local smoothing and maximal
function spaces are excluded.
Then
\begin{align}
2^{k - j} (1 + s 2^{2j})^{8}
\lVert P_j \psi_\ell(s) \cdot P_k \psi_m(0) \rVert_{L^2_{t,x}}^2
\lesssim 2^{-2 \sigma k}
c_j^2 c_k^2(\sigma) + 
\varepsilon^2
b_j^2 b_k^2(\sigma)
\label{BSs}
\end{align}
\label{T:BilinearStrichartz}
\end{thm}
We show how to apply this theorem in the next section.

\section{Proofs of the main theorems} \label{S:Main}

We have in place all of the estimates needed to prove (\ref{Goal})
and hence Theorem \ref{T:EnvelopeTheorem}.

Using the main linear estimate of Proposition \ref{MainLinearEstimate}
and the decomposition introduced in \S \ref{SS:DMP},
we write
\begin{equation}
\begin{split}
\lVert P_k \psi_m \rVert_{G_k(T)}
\lesssim& \;
\lVert P_k \psi_m(0) \rVert_{L^2_x} +
\lVert P_k V_m \rVert_{N_k(T)} \\
&\; + \lVert P_k \Bhi \rVert_{N_k(T)} +
\lVert P_k \Blo \rVert_{N_k(T)}
\end{split}
\end{equation}
It was shown in the two preceding sections that $P_k V_m$
and $P_k \Bhi$ are perturbative in the sense that
\begin{equation*}
\lVert P_k V_m \rVert_{N_k(T)}
+
\lVert P_k \Bhi \rVert_{N_k(T)}
\lesssim \varepsilon b_k
\end{equation*}
where $\varepsilon > 0$ is assumed to satisfy $b_k \leq \varepsilon$
and $\varepsilon^{1/2} \sum_j b_j^2 \ll 1$.
In view of (\ref{iED}) and the bootstrap hypothesis (\ref{Main Bootstrap}),
we set $\varepsilon := \varepsilon_0^{9/10}$. Clearly this satisfies
$b_k \leq \varepsilon$. Moreover,
\begin{equation*}
\varepsilon^{1/2} \sum_j b_j^2
\leq \varepsilon_0^{3/10} \sum_j c_j^2
\lesssim_{E_0} \varepsilon_0^{3/10}
\end{equation*}
To handle $P_k \Blo$, we first rewrite it as
\begin{equation*}
P_k \Blo
=
- i \partial_\ell( \Alol P_k \psi_m)
+ R
\end{equation*}
where $R$ is a perturbative remainder (see \cite[\S 5]{Sm10} for details).
Therefore
\begin{equation}
\lVert P_k \psi_m \rVert_{G_k(T)}
\lesssim
c_k + \varepsilon b_k
+ \lVert \partial_\ell( \Alol P_k \psi_m) \rVert_{N_k(T)}
\label{psi G preliminary}
\end{equation}
Thus it remains to control
$- i \partial_\ell( \Alol P_k \psi_m)$,
which we expand as
\begin{equation}
\begin{split}
- i P_k \partial_\ell \sum_{ \substack{k_1, k_2 \leq k - \varpi \\ \lvert k_3 - k \rvert \leq 4}}
\int_0^\infty \Im(\overline{P_{k_1} \psi_\ell} P_{k_2} \psi_s)(s^\prime)
P_{k_3} \psi_m(0) ds^\prime
\end{split}
\label{lo expanded}
\end{equation}
and whose $N_k(T)$ norm we denote by $\Nlo$.
The key now is to apply Theorem \ref{T:BilinearStrichartz}
to $\overline{P_{k_1} \psi_\ell}(s^\prime)$ and $P_{k_3} \psi_m(0)$,
after first placing all of (\ref{lo expanded}) in $N_k(T)$ using
(\ref{bilin2 improv}).
We obtain
\begin{equation*}
\begin{split}
\Nlo
&\lesssim \; 
2^k \sum_{ \substack{k_1, k_2 \leq k - \varpi \\ \lvert k_3 - k \rvert \leq 4}}
2^{- \lvert k - k_2 \rvert / 2} 2^{- \lvert k_1 - k_3 \rvert / 2} 2^{-\max\{k_1, k_2\}} b_{k_2}
\left(c_{k_1} c_{k_3} + \varepsilon b_{k_1} b_{k_3} \right) 
\\
&\lesssim \;
2^k \sum_{k_1, k_2 \leq k - \varpi} 2^{(k_1 + k_2)/2 - k} 2^{-\max\{k_1, k_2\}}
b_{k_2}(c_{k_1} c_k + \varepsilon b_{k_1} b_k)
\end{split}
\end{equation*}
Without loss of generality we restrict the sum to $k_1 \leq k_2$:
\begin{equation*}
\sum_{k_1 \leq k_2 \leq k - \varpi} 2^{(k_1 - k_2)/2} b_{k_2}(c_{k_1} c_k + \varepsilon b_{k_1} b_k)
\end{equation*}
Using the frequency envelope property to sum off the diagonal, we reduce to
\begin{equation*}
\Nlo
\lesssim
\sum_{j \leq k - \varpi} (b_j c_j c_k + \varepsilon b_j^2 b_k)
\end{equation*}
Combining this with (\ref{psi G preliminary}) and the fact that
$R$ is perturbative,
we obtain
\begin{equation*}
b_k \lesssim c_k + \varepsilon b_k +
\sum_{j \leq k - \varpi} (b_j c_j c_k + \varepsilon b_j^2 b_k)
\end{equation*}
which, in view of our choice of $\varepsilon$, reduces to
\begin{equation*}
b_k \lesssim c_k + c_k \sum_{j \leq k - \varpi} b_j c_j 
\end{equation*}
Squaring and applying Cauchy-Schwarz yields
\begin{equation}
b_k^2 \lesssim  (1 +  \sum_{j \leq k - \varpi} b_j^2 ) c_k^2
\label{SCS}
\end{equation}
Setting
\begin{equation*}
B_k := 1 + \sum_{j < k} b_j^2
\end{equation*}
in (\ref{SCS}) leads to
\begin{equation*}
B_{k+1} \leq B_k (1 + C c_k^2)
\end{equation*}
with $C > 0$ independent of $k$.
Therefore
\begin{equation*}
B_{k + m} 
\leq 
B_k \prod_{\ell = 1}^m (1 + C c_{k + \ell}^2)
\leq
B_k \exp(C \sum_{\ell = 1}^m c_{k + \ell}^2)
\lesssim_{E_0} B_k
\end{equation*}
Since $B_k \to 1$ as $k \to -\infty$, we conclude
\begin{equation*}
B_k \lesssim_{E_0} 1
\end{equation*}
uniformly in $k$, so that, in particular,
\begin{equation}
\sum_{j \in \mathbb{Z}} b_j^2
\lesssim 1
\label{b l2 bounded}
\end{equation}
which, joined with (\ref{SCS}), implies (\ref{Goal}).

The proof of (\ref{Goal2}) is almost an immediate consequence.
A bit of care must be exercised, though. For instance,
we will have summations of terms such as
$C_{k, k_1, k_2, k_3} b_{j_1} b_{j_2} b_{j_3}(\sigma) 2^{-\sigma j_3}$
to control (see for instance Corollary \ref{C:MTE1}), where
$\{j_1, j_2, j_3\}$ is some permutation of $\{k_1, k_2, k_3\}$.
Clearly such a term does not sum over $j_3 \ll C$.
The way out, however, is straightforward.
Recall that we need only sum over $Z_1(k), Z_2(k), Z_3(k)$,
defined in (\ref{Z def}).
If we encounter a sum over $Z_1(k)$, bound the $k_3$ term
with $b_{k_3}(\sigma) 2^{-\sigma k_3}$ and the remaining two
terms with $b_{k_1}$ and $b_{k_2}$. Over $Z_2(k)$, always
bound the $k_1$ term by $b_{k_1}(\sigma) 2^{-\sigma k_1}
\leq b_{k_1}(\sigma) 2^{-\sigma k}$.
Similarly, on $Z_3(k)$ we also bound the $k_1$ term by
$b_{k_1}(\sigma) 2^{-\sigma k_1}$.
Such a strategy suffices for controlling the perturbative terms.

We obtain
\begin{equation*}
b_k(\sigma) \lesssim c_k(\sigma) + \varepsilon b_k(\sigma)
+ \sum_{j \leq k - \varpi} (b_j c_j c_k(\sigma) + \varepsilon b_j^2 b_k(\sigma))
\end{equation*}
which is enough to prove (\ref{Goal2}) in view of (\ref{b l2 bounded}).

Equipped with Theorem \ref{T:EnvelopeTheorem}, we conclude
Corollary \ref{C:ETcor} as in \cite[\S4.6]{Sm10}.

\bibliography{SMbib}
\bibliographystyle{amsplain}

\end{document}